\documentclass[a4paper,11pt]{article}
\usepackage{benstyle}

\usepackage{hyperref,verbatim,float,stackrel}

\usepackage[normalem]{ulem}

\DefineNamedColor{named}{Purple}{cmyk}{0.65,0.95,0,0}

\DefineNamedColor{named}{Orange}{rgb}{1,0.55,0}

\title{Weighted frames of exponentials and stable recovery of multidimensional functions from nonuniform Fourier samples}
\author{Ben Adcock\footnote{Department of Mathematics, Simon Fraser University, BC V5A 1S6, Canada (ben\_adcock@sfu.ca). The author was supported by the NSF DMS grant 1318894.} \and Milana Gataric\footnote{CCA, Centre for Mathematical Sciences, University of Cambridge, CB3 0WA, UK (m.gataric@maths.cam.ac.uk). The author was supported by the UK EPSRC grant EP/H023348/1 for the University of Cambridge Centre for Doctoral Training, the Cambridge Centre for Analysis.} \and Anders C. Hansen \footnote{DAMTP, Centre for Mathematical Sciences, University of Cambridge, CB3 0WA, UK (ach70@cam.ac.uk). The author was supported by a Royal Society University Research Fellowship as well as the EPSRC grant EP/L003457/1.}}

\begin{document}

\maketitle

%
%

\vspace{1pc}

\begin{abstract}
In this paper, we consider the problem of recovering a compactly supported multivariate function from a collection of pointwise samples of its Fourier transform taken nonuniformly. We do this by using the concept of weighted Fourier frames. A seminal result of Beurling shows that sample points give rise to a classical Fourier frame provided they  are relatively separated and of sufficient density.  However, this result does not allow for arbitrary clustering of sample points, as is often the case in practice.  Whilst keeping the density condition sharp and dimension independent, our first result removes the separation condition and shows that density alone suffices.  However, this result does not lead to estimates for the frame bounds.  A known result of Gr\"ochenig provides explicit estimates, but only subject to a density condition that deteriorates linearly with dimension.  In our second result we improve these bounds by reducing the dimension dependence. In particular, we provide explicit frame bounds which are dimensionless for functions having compact support contained in a sphere.  Next, we demonstrate how our two main results give new insight into a reconstruction algorithm---based on the existing generalized sampling framework---that allows for stable and quasi-optimal reconstruction in any particular basis from a finite collection of samples.  Finally, we construct sufficiently dense sampling schemes that are often used in practice---jittered, radial and spiral sampling schemes---and provide several examples illustrating the effectiveness of our approach when tested on these schemes.
\end{abstract}

\begin{small}
\hspace{0.22cm} \textbf{Key words.} Fourier frames, nonuniform sampling, generalized sampling, medical imaging
\end{small}

\section{Introduction}

The recovery of a compactly supported function from pointwise measurements of its Fourier transform---or equivalently, the recovery of a band-limited function from its  direct samples---has been the subject of comprehensive research during the past century,  driven by numerous practical applications ranging from Magnetic Resonance Imaging (MRI) to  Computed Tomography (CT), geophysical imaging, seismology and electron microscopy.  In many of these applications, the case when the data is acquired nonuniformly is of particular interest. For instance, MR scanners often use spiral sampling geometries for fast data acquisition.  Such sampling geometries are often preferable because of the higher resolution obtained in the Fourier domain and the lower magnetic gradients required to scan along such trajectories.   Another important example is radial (also known as polar) sampling of the Fourier transform, which is used in MRI, as well as in applications where the Radon transform is involved in the sampling process; CT, for instance.  For examples of different sampling patterns used in applications see Figure \ref{fig:ss}. Spurred by its practical importance, the past decades have witnessed the development of an extensive mathematical theory of nonuniform sampling, as evidenced by a vast body of literature.  An inexhaustive list includes the books of Marvasti \cite{Marvasti}, Benedetto and Ferreira \cite{BenedettoBook}, Young \cite{Young}, Seip \cite{SeipBook} and others, as well as many excellent articles; see \cite{AldroubiGrochenigSIREV,  BenedettoFrames, BenedettoSpiral,  FeichtingerGrochenigIrregular, FeichtingerEtAlEfficientNonuniform, GrochenigStrohmerMarvasti,  StrohmerNonuniformNA} and references therein.

In the case of Cartesian sampling, the celebrated Nyquist--Shannon theorem \cite{Unser50years} guarantees a full reconstruction of a compactly supported signal from its Fourier measurements, provided that the samples are taken equidistantly at a sufficiently high rate, equal to or exceeding the so-called Nyquist rate.  In other words, the samples must be taken uniformly and densely enough.  Nonuniform sampling is typically studied within the context of so-called Fourier frames.  The theory of Fourier frames was developed by Duffin and Shaeffer \cite{DuffSch}, more than half a century ago, and its roots can be traced back to earlier works of Paley and Wiener \cite{PaleyWiener} and Levinson \cite{Levinson}. In one dimension, there exists a near-complete characterization of Fourier frames in terms of the density of underlying samples, due primarily to Beurling \cite{BeurlingDiffOp}, Landau \cite{Landau}, Jaffard \cite{Jaffard} and Seip \cite{Seip1995}. However, in higher dimensions, the situation becomes considerably more complicated \cite{BenedettoSpiral, OlevskiiUlanovskii}. Nevertheless, Beurling's seminal paper \cite{BeurlingDiffOp} (see also \cite{BeurlingVol2}) provides a sharp sufficient condition for sampling points in multiple dimensions to give rise to a Fourier frame for the space of $\rL^2$ functions  compactly supported on a sphere. This was generalized to the spaces of $\rL^2$ functions compactly supported on any compact, convex and symmetric set by Benedetto and Wu \cite{BenedettoSpiral} (see also the work by Olevskii and Ulanovskii \cite{OlevskiiUlanovskii}).  Regarding general bounded supports in $\bbR^d$, Landau \cite{Landau} provides a necessary density condition that fails to be sufficient in general. A  recent result due to Matei and Meyer \cite{quasi} proves this density condition to be sufficient in the special case of sampling on quasicrystals. Also, some of these density-type results where extended to shift-invariant spaces by Aldroubi and Gr\"ochenig \cite{AldroubiGrochenig_SIdensity}.  However, in our work, we focus on compactly supported and square-integrable functions with supports in $\bbR^d$  which are  compact, convex and symmetric. For a more detailed review on the theory of Fourier frames and nonuniform sampling see \cite{AldroubiGrochenigSIREV, BenedettoSpiral, ChristensenFramesAMS}.

\begin{figure}
\begin{center}
\includegraphics[scale=0.8, trim=-0cm -0cm -0.cm 0cm]{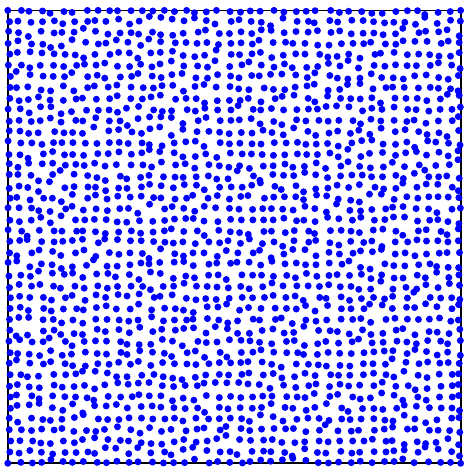}   \includegraphics[scale=0.8, trim=-0cm -0.cm -0cm 0.cm]{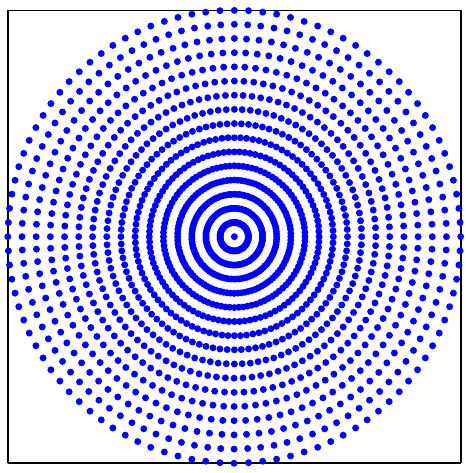}
\includegraphics[scale=0.8, trim=-0cm -0cm -0cm -0.cm]{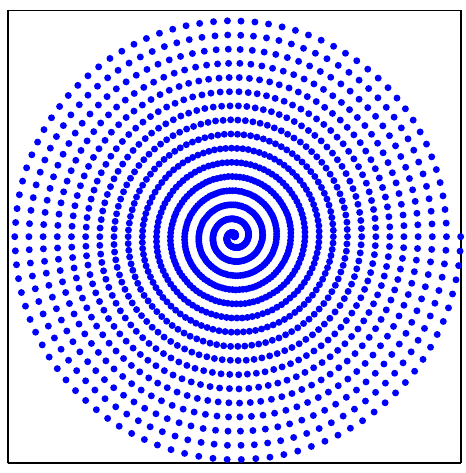}   \includegraphics[scale=0.8, trim=-0cm -0cm -0cm -0.cm]{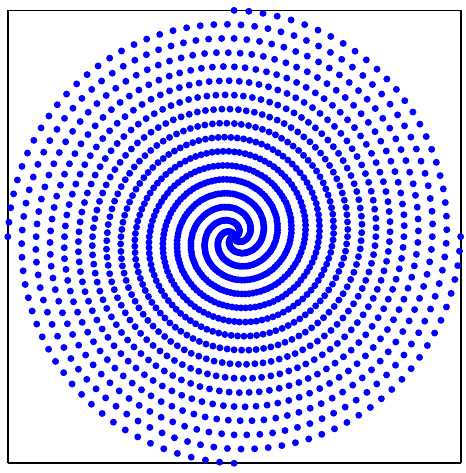}
\vspace{-0.5cm}
\end{center}
\caption{\small{Different sampling schemes:  (i) \textit{jittered sampling scheme},  a standard model when the measurements are not taken exactly on a uniform grid, often used MRI, seismology and geophysics \cite{AldroubiGrochenigSIREV, Marvasti}, (ii) \textit{polar sampling scheme} used in computed tomography \cite{Epstein}, (iii) \textit{spiral} and (iv) \textit{interleaving spiral} used in MRI \cite{spyralScienceDirect}. All of them satisfy an appropriate $(K,\delta_{E^\circ})$-density condition (see Definition \ref{K-dense}), for $E=[-1,1]^2$, $\delta_{E^\circ}<0.25$ and $K=4$.  } 
}
\label{fig:ss}
\end{figure}

\subsection{Main results of this paper}

A limitation of the results mentioned above is that they require a minimal separation between the sampling points.  In particular, clustering of sampling points deteriorates the associated frame bounds, which leads to numerical instability.  The main contribution of the first part of this paper removes the minimal separation restriction whilst keeping the sharpness of the result.  Through the use of a weighted Fourier frame approach, based on Gr\"ochenig's earlier work (see below), we adapt Beurling's result to allow for arbitrary clustering of sampling points.  Specifically, we prove the following:

\begin{theorem}\label{t:weighted_frame} 
Let $\rH= \{ f \in \rL^2(\bbR^d) : \textnormal{supp}(f) \subseteq E \}$,  
where $E\subseteq \mathbb{R}^d$ is compact, convex and symmetric.
If a countable set $\Omega\subseteq\hat{\bbR}^d$ has density  $\delta_{E^\circ}<1/4$  $($see Definition $\ref{d:delta-dense}$$)$ then
there exist weights $\mu_{\omega} > 0$ such that $\{\sqrt{\mu_\omega} e_\omega\}_{\omega\in\Omega}$  is a weighted Fourier frame for $\rH$, where $e_{\omega}(x) = \E^{\I 2 \pi \omega \cdot x} {\bf1}_{E}(x)$.  In other words,
there exist constants $A,B>0$ such that 
\bes{
\forall f\in \rH, \quad A\|{f}\|^2 \leq \sum_{\omega\in\Omega}\mu_\omega|\hat{f}(\omega)|^2 \leq B \|{f}\|^2.
}
In particular, it suffices to choose the weights $\{\mu_\omega\}_{\omega\in\Omega}$ as the measures of Voronoi regions $($see Definition $\R{Voronoi}$$)$ with respect to norm $\abs{\cdot}_{E^\circ}$ $($see $\R{norm32}$ and $\R{polar}$$)$.
\end{theorem}

The $1/4$ density condition given here is sharp: if a countable set $\Omega$ does not satisfy the required density condition, then the associated family of weighted exponentials  $\{\sqrt{\mu_\omega} e_\omega\}_{\omega\in\Omega}$ does not have to give a weighted Fourier frame with the weights chosen as in Theorem \ref{t:weighted_frame}.

This result has both theoretical and practical significance.  First, it is interesting to address the issue of arbitrary clustering, since it is natural to anticipate that adding more sampling points should not impair the recovery of a function.
Second, this scenario often arises in applications.  For example, consider Fourier measurements acquired on a polar sampling scheme. By increasing the number of radial lines along which samples are acquired, the sampling points cluster at low frequencies, which deteriorates the frame bounds of the corresponding Fourier frame.  On the other hand, if we weight those points according to their relative densities, the resulting \textit{weighted} Fourier frame has controllable frame bounds.

Weighted Fourier frames, which we also refer to as weighted frames of exponentials, were studied by Gr\"ochenig \cite{GrochenigIrregular}, and later also by Gabardo \cite{Gabardo}. In \cite{GrochenigIrregular}, Gr\"ochenig presents a sufficient density condition in order for a family of exponentials to constitute a weighted Fourier frame, and provides explicit frame bounds.  This density condition is sharp in dimension $d=1$, but fails to be sharp in higher dimensions, with the  estimate on the density deteriorating linearly, and the estimates on the frame bounds, exponentially in $d$. The multidimensional result has been improved in \cite{BassGrochenigRandom}, but  under the assumption that the sampling set consists of a sequence of uniformly distributed independent random variables. In this setting, Bass and Gr\"ochenig provide probabilistic estimates.

Our work focuses on deterministic statements and provides two improvements of Gr\"ochenig's result from  \cite{GrochenigIrregular}.  First, as discussed above, in Theorem \ref{t:weighted_frame} we provide a density condition which is both sharp \textit{and} dimensionless.  Unfortunately, however, this condition does not give rise to explicit frame bounds.  Therefore, in our second result we present explicit frame bounds under a less stringent density condition than previously known:

\begin{theorem}\label{t:Groch_improved} 
Let $\rH= \{ f \in \rL^2(\bbR^d) : \textnormal{supp}(f) \subseteq E \}$, where $E\subseteq \mathbb{R}^d$ is compact.  Suppose that $\abs{\cdot}_*$ is an arbitrary norm on $\bbR^d$ and $c^*>0$ is the smallest constant for which $\abs{\cdot}\leq c^* \abs{\cdot}_*$, where $\abs{\cdot}$ denotes the Euclidean norm.
Let $\Omega\subseteq\bbR^d$ be  $\delta_*$-dense $($see Definition $\ref{d:delta-dense}$$)$ with
\be{\label{delta_improved}
\delta_*<  \frac{\log2}{2 \pi m_E c^*},
} 
where $m_E = \sup_{x\in E}|x|$. Then $\{\sqrt{\mu_\omega}e_\omega\}_{\omega\in\Omega}$ is a weighted Fourier frame for $\rH$ with the weights defined as the measures of Voronoi regions with respect to norm $\abs{\cdot}_*$.  The weighted Fourier frame bounds $A,B > 0$ satisfy   
\eas{
\sqrt{A} \geq 2-\exp(2 \pi m_E \delta_* c^*) ,   \qquad   \sqrt{B} \leq \exp(2 \pi m_E \delta_* c^*) < 2.
}
\end{theorem}

Taking $\abs{\cdot}_{*} = \abs{\cdot}$ for simplicity, where $\abs{\cdot}$ is the Euclidean norm, we see that the key estimate \R{delta_improved}, which is a refinement of Gr\"ochenig's, deteriorates with dimension only for certain function supports $E$.  Specifically, it depends on the radius of the largest sphere in which $E$ is contained, i.e. it depends on $m_E$.  In particular, \R{delta_improved} is dimensionless when a function has a compact support contained in the unit Euclidean ball $\cB_{1}$, since then $m_E=1$.  In this case, Theorem \ref{t:weighted_frame} gives the sharp sufficient condition $\delta < 0.25$ (where $\delta$ corresponds to the Euclidean norm) but without explicit frame bounds.  On the other hand, Theorem \ref{t:Groch_improved} provides explicit frame bounds under the slightly stronger, but dimension independent, condition $\delta < \frac{\log2}{2 \pi} \approx 0.11$.

We note at this stage that, whilst Gr\"ochenig was arguably the first to rigorously study weighted Fourier frames in sampling, the use of weights is commonplace in MRI reconstructions, where they are often referred to as `density compensation factors' (see \cite{spyralScienceDirect, FesslerFastIterativeMRI} and references therein).  However, such approaches are often heuristic.  Building on Gr\"ochenig's earlier work, our results provide further mathematical theory supporting their use.

In practice, one only has access to a finite number of samples.  In the final part of this paper, we consider a reconstruction algorithm for this problem, based on the generalized sampling (GS) framework introduced in \cite{BAACHShannon} (see also \cite{BAACHAccRecov,BAACHOptimality,AHPWavelet}).  In particular, in \S \ref{s:NUGS}, we give the third main result of this paper, Theorem \ref{t:admissible},  which shows that stable, quasi-optimal reconstruction is possible in \textit{any} subspace $\rT\subseteq\rH$ provided the samples satisfy the same density conditions as in Theorems \ref{t:weighted_frame} and \ref{t:Groch_improved}, and additionally, provided the samples possess a sufficiently large bandwidth, in a sense we define later. Hence, we extend the analysis of the framework considered in \cite{1DNUGS}---so-called nonuniform generalized sampling (NUGS)---to the multidimensional setting.

We also remark that our analog recovery model is the same as that used with great success in the recent work of Guerquin-Kern, Haberlin, Pruessmann and Unser \cite{PruessmannUnserMRIFast} on iterative, wavelet-based reconstructions for MRI. Moreover, the popular iterative reconstruction algorithm of Sutton, Noll and Fessler \cite{FesslerFastIterativeMRI} for non-Cartesian MRI is a special case of NUGS based on a digital signal model. Therefore, the results we prove in this paper provide theoretical foundations for the success of those algorithms as well.  Our results also improve existing bounds for the well-known ACT (Adaptive weights, Conjugate gradients, Toeplitz) algorithm in nonuniform sampling \cite{FeichtingerGrochenigIrregular,FeichtingerEtAlEfficientNonuniform,GrochenigModernSamplingBook,GrochenigStrohmerMarvasti}, which can also be viewed as a particular case of NUGS. For further discussion, see \S \ref{ss:implementation_and_relation} of this paper.

The remainder of this paper is organized as follows.  In \S\ref{s:weighted_frames}  we consider weighted Fourier frames and the proofs of Theorems \ref{t:weighted_frame} and \ref{t:Groch_improved}.  We discuss the NUGS framework in \S \ref{s:NUGS}, and show stable and accurate recovery by using the results from \S \ref{s:weighted_frames}.  Next in \S \ref{s:dense_ss}  we construct several popular sampling schemes so that they satisfy appropriate density conditions.  Finally, we illustrate our theoretical results in \S \ref{s:num_results} with some numerical experiments.

\section{Weighted frames of exponentials} \label{s:weighted_frames}

\subsection{Background material and preliminaries} \label{ss:background}

Let 
\bes{
\rH= \left\{ f \in \rL^2(\bbR^d) : \textnormal{supp}(f) \subseteq E \right\}
}
be the Hilbert space of square-integrable functions supported on a compact set $E\subseteq\mathbb{R}^d$, with  the standard $\rL^2$-norm $\nm{\cdot}$ and $\rL^2$-inner product $\ip{\cdot}{\cdot}$. The $d$-dimensional Euclidean vector space is denoted by $\mathbb{R}^d$, and, following a standard convention, $\mathbb{\hat{R}}^d$  is used whenever $\mathbb{R}^d$ is considered as a frequency domain. For $f\in\rH$,  the Fourier transform is defined by
\bes{
\hat{f}(\omega) =\int_{E} f(x) \E^{-\I  2\pi \omega \cdot x} \D x,\quad \omega \in \mathbb{\hat{R}}^d,
}
where $\cdot$ stands for the Euclidean inner product. We also use the following notation
\bes{
e_{\omega}(x) = \E^{\I 2 \pi \omega \cdot x} {\bf1}_{E}(x),
}
where ${\bf1}_{E}$ is the indicator function of the set $E$. Note that $\hat{f}(\omega) = \ip{f}{e_\omega}$.

Let $\abs{\cdot}_*$ denote an arbitrary norm  on $\bbR^d$.  Note that for every such norm the set $\{x\in\bbR^d:|x|_*\leq1\}$ is convex, compact and symmetric. Moreover, all norms on a finite-dimensional space are equivalent to the Euclidean norm, which we denote simply by $\abs{\cdot}$. Hence,  by $c_*,c^*>0$, we denote the sharp constants for which
\bes{
\forall x\in\bbR^d, \quad c_*|x|_* \leq |x|\leq c^*|x|_*.
}

Conversely, if $E\subseteq\mathbb{R}^d$ is a compact, convex and symmetric set, the function $\abs{\cdot}_E:\mathbb{R}^d\rightarrow \mathbb{R}$ defined by 
\be{\label{norm32}
\forall x\in\bbR^d, \quad |x|_E=\inf\{a>0 : x\in aE\},
}
is a norm on  $\mathbb{R}^d$ \cite{BenedettoSpiral}. Here, $E$ is the unit ball with the respect to the norm $\abs{\cdot}_E$, i.e. 
\bes{
E=\{x\in\bbR^d:|x|_E\leq1\}.
}
Also, for such set $E\subseteq\mathbb{R}^d$, 
its polar set is defined as
\be{\label{polar}
E^\circ=\{ \hat y \in \mathbb{\hat{R}}^d : \forall x \in E,\ x\cdot\hat y\leq1 \}.
}
Note that $E^\circ$ is itself a convex, compact and symmetric set in $\mathbb{\hat{R}}^d$, which is the unit ball with respect to the norm $\abs{\cdot}_{E^\circ}$. Also observe that,  if $E$ is the unit ball in the Euclidean norm, which we denote by $\cB_1$, then $\cB_1=\cB_1^{\circ}$ and $\abs{\cdot}_{\cB_1}=\abs{\cdot}_{\cB_1^{\circ}}=\abs{\cdot}$.

Throughout the paper, we denote $\ell^p$-norm by $\abs{\cdot}_p$, i.e.\ for $x\in \bbR^d$, $|x|_p=\left(\sum_{j=1}^d |x_j|^p\right)^{1/p}$. Hence $\abs{\cdot}_2=\abs{\cdot}_{\cB_1}=\abs{\cdot}$. Also, we recall the well-know inequality
\be{\label{norms_ineq}
\forall x \in \bbR^d, \quad |x|_{q}\leq|x|_{r}\leq d^{1/r-1/q} |x|_{q}, \quad q>r>0.
}

Now, let $\Omega\subseteq\mathbb{\hat{R}}^d$ be a countable set of sampling points, which we also refer to as a \textit{sampling scheme}. The set $\Omega$ is said to be \textit{separated} with respect to the $\abs{\cdot}_{*}$-norm if 
there exists a constant $\eta>0$ such that
\bes{
\forall \omega, \lambda \in \Omega , \quad \omega \neq \lambda, \quad |\omega-\lambda|_{*} \geq \eta,
}
and it is \textit{relatively separated} if it is a finite union of separated sets.
It is clear that, if $\Omega$ is separated in the $\abs{\cdot}_{*}$-norm then it is separated in any norm on $\mathbb{\hat{R}}^d$ and vice-versa.

Next, we introduce the crucial notion of \textit{density} of a countable set $\Omega\subseteq\mathbb{\hat{R}}^d$. This definition originates in Beurling's work \cite{BeurlingDiffOp} and it is used frequently in multidimensional nonuniform sampling literature. 

\begin{definition}\label{d:delta-dense}
Let $\Omega$ be a sampling scheme contained in a closed, simply connected set $Y \subseteq\mathbb{\hat{R}}^d$ with $0$ in its interior. Let $\abs{\cdot}_*$ be an arbitrary norm on $\bbR^d$, and let $\delta_*>0$.  We say that $\Omega$ is $\delta_*$-dense in the domain $Y$  if
\bes{
\delta_*=\sup_{\hat{y}\in Y} \inf_{\omega\in \Omega}|\omega-\hat{y}|_*.
}
If $\abs{\cdot}_{*} = \abs{\cdot}_{E}$ for a compact, convex and symmetric set $E$, then we write $\delta_{E}$. Also, to emphasise the sampling scheme, where necessary we use notation $\delta_*(\Omega)$.
\end{definition}

Note that the $\delta_*$-density condition from the Definition \ref{d:delta-dense} is equivalent to the \textit{$\delta_*$-covering condition}: there exists $\delta_*\in(0,1/4)$ such that for all  $\rho\geq \delta_*$ it holds that
\bes{
Y \subseteq \bigcup_{\omega\in\Omega}\left\{x\in\bbR^d:|x-\omega|_*\leq \rho \right\}.
}

Before we define weighted frames, let us discuss classical frames of exponentials.
A countable family of functions $\{e_\omega\}_{\omega\in\Omega}\subseteq\rH$  is said to be a \textit{Fourier frame} for $\rH$ if there exist constants $A,B>0$ such that 
\be{\label{frame}
\forall f\in \rH, \quad A\|{f}\|^2 \leq \sum_{\omega\in\Omega}|\hat{f}(\omega)|^2 \leq B \|{f}\|^2.
}
The constants $A$ and $B$ are called \textit{upper and lower frame bounds}, respectively. If $\{e_\omega\}_{\omega\in\Omega}$ is a frame, then the \textit{frame operator} $\cS:\rH\rightarrow \rH$ is defined by 
\bes{
\forall f\in\rH,\quad \cS : f \mapsto \cS f = \sum_{\omega \in \Omega} \hat{f}(\omega) e_{\omega}.
}
Since the inequality \R{frame} holds, the frame operator $\cS$ is a topological isomorphism with the inverse $\cS^{-1}:\rH\rightarrow \rH$, and also 
\be{\label{rec_formula}
\forall f\in\rH, \quad f = \sum_{\omega \in \Omega} \ip{\cS^{-1} f }{e_{\omega}} e_{\omega}.
}
Formula \R{rec_formula}, with the appropriately truncated sum, is sometimes used for signal reconstruction \cite{BenedettoSpiral}.  However, for the types of sets $\Omega$ considered in practice, finding the inverse frame operator $\cS^{-1}$ is often a nontrivial task.  Typically, this renders such an approach infeasible in more than one dimension.  

If the relation \R{frame} holds with $A=B$, the family $\{e_\omega\}_{\omega\in\Omega}$ is called a \textit{tight frame}, and if $A=B=1$, this family forms  an orthonormal basis for $\rH$. In these cases, the relation \R{frame} is  known as (generalized) Parseval's equality. Also, then the frame operator becomes  $\cS= A\cI$, where $\cI$ is the identity operator on $\rH$, and the formula \R{rec_formula} represents the Fourier series of $f$.  Moreover, the appropriately truncated Fourier series converges to $f$ on $\rH$. This leads to a considerably simpler framework in the case when the samples are acquired uniformly, corresponding to an orthonormal basis or a tight frame  for $\rH$.

In \cite{BeurlingDiffOp}, Beurling provides a sufficient density condition for a nonuniform set of sampling points to give a Fourier frame for $\rH$ consisting of functions supported on the unit sphere in the Euclidean norm. In what follows, we use a variation of Beurling's result given by Benedetto \& Wu in \cite{BenedettoSpiral}, and also by Olevskii \& Ulanovskii \cite{OlevskiiUlanovskii}, which is a generalization to arbitrary convex, compact and symmetric domains:

\begin{theorem}
\label{t:Beurling_thm}
Let $E\subseteq\bbR^d$ be compact, convex and symmetric set. If  $\Omega\subseteq\mathbb{\hat{R}}^d$ is  relatively separated and $\delta_{E^\circ}$-dense  in the domain  $Y = \mathbb{\hat{R}}^d$ with $\delta_{E^\circ}<1/4$,
then $\{e_\omega\}_{\omega\in\Omega}$ is a Fourier frame for $\rH$.
\end{theorem}

Beurling  \cite{BeurlingDiffOp} also shows that this result is sharp in the sense that there exists a countable set with the density $\delta_{E^\circ}=1/4$, where $E$ is the unit ball in the Euclidean metric, which does not satisfy the lower frame condition in \R{frame} (see also \cite[Prop.\ 4.1]{OlevskiiUlanovskii}).

Now we define weighted frames of exponentials:

\begin{definition}\label{weighted_Fourier}
A countable family of functions $\{\sqrt{\mu_\omega}e_\omega\}_{\omega\in\Omega}$  is a weighted Fourier frame for $\rH$, with weights $\{ \mu_\omega \}_{\omega \in \Omega}$, $\mu_\omega > 0$, if there exist constants $A,B>0$ such that 
\be{\label{w_frame}
\forall f\in \rH, \quad A\|{f}\|^2 \leq \sum_{\omega\in\Omega}\mu_\omega|\hat{f}(\omega)|^2 \leq B \|{f}\|^2.
}
\end{definition}

As discussed, the use of weights  is to compensate for arbitrary clustering in $\Omega$.  In order to define appropriate weights  $\{\mu_\omega\}_{\omega\in\Omega}$  corresponding to the sampling scheme $\Omega$, in this paper, we use measures of Voronoi regions.  This is a standard practice in nonuniform sampling \cite{AldroubiGrochenigSIREV, Rasche99}. 

\begin{definition}\label{Voronoi}
Let $\Omega$ be a set of distinct points in  $Y \subseteq\mathbb{\hat{R}}^d$ and let $\abs{\cdot}_*$ be an arbitrary norm on $\bbR^d$. The Voronoi region at $\omega\in\Omega$, with respect to the  norm $\abs{\cdot}_*$ and in the domain $Y$, is given by
\bes{
V^{*}_\omega = \left \{ \hat{y} \in Y : \forall \lambda \in \Omega,\ \lambda\neq \omega ,\ |\omega-\hat{y}|_{*} \leq |\lambda-\hat{y}|_{*} \right\}.
}
\end{definition}
The Lebesgue measure of the Voronoi region $V^{*}_\omega$ we denote as
\bes{
\text{meas}\left(V^{*}_\omega\right)=\int_{Y}{\bf1}_{V^{*}_\omega}(\hat{y})\D \hat{y}. 
}

In \cite{GrochenigIrregular}, Gr\"ochenig provides explicit frame bounds for weighted Fourier frames, provided the sample points $\Omega$ are sufficiently dense. In one dimension, the condition on the density is sharp, i.e., sampling points with density such that $\delta<1/4$ give rise to a weighted Fourier frame, but sets of points with lower density (i.e.\ bigger delta) do not necessarily yield a weighted Fourier frame. However, the sharpness of the result is lost in higher dimensions.

Here we state Gr\"ochenig's multidimensional result \cite[Prop.\ 7.3]{GrochenigModernSamplingBook}, which is a more recent reformulation of \cite[Thm.\ 5]{GrochenigIrregular}:

\begin{theorem}\label{t:Grocg_orig} 
Let $\rH=\{ f \in \rL^2(\bbR^d) : \textnormal{supp}(f) \subseteq E \}$, where  $E=[-1,1]^d$. If $\Omega\subseteq\mathbb{\hat{R}}^d$ is a $\delta_{\cB_1}$-dense set of distinct points such that
\be{\label{Grochenig_delta}
\delta_{\cB_1}   < \frac{\log2}{2\pi d},
}
then $\{\sqrt{\mu_\omega} e_\omega\}_{\omega\in\Omega}$ is a weighted Fourier frame for $\rH$, where the weights are defined as measures of the Voronoi regions of the points $\Omega$ with respect to the Euclidean norm. The weighted frame bounds $A,B > 0$ satisfy   
\eas{
\sqrt{A} \geq  2-\E^{2\pi\delta_{\cB_1} d} ,\qquad \sqrt{B} \leq   \E^{2\pi\delta_{\cB_1} d} < 2.
} 
\end{theorem}

Note that the bound \R{Grochenig_delta} deteriorates linearly with the dimension $d$. Also, $E$ can be any rectangular domain of the form  $\prod_{i=1}^d [-s_i,s_i]$, since $\text{supp}(f)\subseteq \prod_{i=1}^d [-s_i,s_i]$ implies that $\tilde{f}(x)=f(x_1/s_1,\ldots,x_d/s_d)$ has support in $[-1,1]^d$. Hence, the result is stated for $E=[-1,1]^d$ without loss of generality \cite{GrochenigModernSamplingBook}. Moreover, note that $E$ may also be any compact set that is a subset of $[-1,1]^d$ such as any $\ell^p$ unit ball, $p>0$, for example.

\subsection{Weighted Fourier frames with explicit frame bounds and the  proof of Theorem \ref{t:Groch_improved}}
 \label{ss:explicit_bounds}

Much like Beurling's result, Theorem \ref{t:Beurling_thm}, it is expected that the density condition for weighted Fourier frames given in Theorem \ref{t:Grocg_orig} does not depend on dimension.  Unfortunately, Gr\"ochenig's estimates deteriorate linearly with the dimension $d$, and thus cease to be sharp. Therefore, in Theorem \ref{t:Groch_improved} we provide an modification of Gr\"ochenig's result by presenting explicit bounds with slower, and sometimes no deterioration with respect to dimension.

The estimates in Theorem \ref{t:Groch_improved} are presented in terms of the following quantity
\bes{
m_E = \sup_{x\in E}|x|,
}
where  $E\subseteq \bbR^d$ and $\abs{\cdot}$ is Euclidean  norm. Note that $m_{\cB_1}=1$ and therefore it is independent of dimension for spheres. Moreover, if  $E$ is the $\ell^p$ unit ball, i.e. $E=\{x:\bbR^d:|x|_{p}\leq 1\}$, $p>0$, then 
\be{\label{me_norms}
m_E=\max \{ 1 , d^{1/2-1/p} \},
}
due to  inequality \R{norms_ineq}.

Let us recall here the multinomial formula. For any $k\in\bbN_0$ and $x\in\bbR^d$, we have
\be{\label{multinomial}
\sum_{|\alpha|_{1}=k}\frac{k!}{\alpha!} x^\alpha = (x_1+\cdots+x_d)^k,
}
where $\alpha=(\alpha_1,\ldots,\alpha_d)$, $|\alpha|_{1}=|\alpha_1|+\ldots+|\alpha_d|$, $\alpha!=\prod_{j=1}^d\alpha_j!$ and $x^\alpha=\prod_{j=1}^d x_j^{\alpha_j}$.
Regarding the multi-index notation, in what follows, we also use the derivative operator defined as
\bes{
D^{\alpha} = \frac{\partial^{|\alpha|_{1}}}{\partial_{x_1}^{\alpha_1}\cdots\partial_{x_d}^{\alpha_d}}.
}

Now we are ready to prove our main result for weighted Fourier frames with explicit bounds, namely Theorem \ref{t:Groch_improved}.

\prf{[Proof of Theorem $\ref{t:Groch_improved}$]
The proof is set up in the same manner as the proof of Gr\"ochenig's original result, Theorem \ref{t:Grocg_orig}. 
For a function $f\in\rH$, define 
\bes{
\chi(\hat{y})=\sum_{\omega\in\Omega}\hat{f}(\omega){\bf1}_{V^*_\omega}(\hat{y}),\quad \hat{y} \in \mathbb{\hat{R}}^d.
}
Since the sets $V^*_\omega$, $\omega\in\Omega$, make a disjoint partition of $\mathbb{\hat{R}}^d$, it holds that
\bes{
\|\chi\|=\sqrt{\sum_{\omega\in\Omega}\mu_\omega|\hat{f}(\omega)|^2},
}
where $\mu_\omega=\text{meas} (V^*_\omega)$.
Note that
\be{\label{step_ineq_00}
\|f\| - \|\hat{f}-\chi\| \leq \|\chi\| \leq \|\hat{f}-\chi\| + \|f\|.
}
Hence, we aim  to estimate $\|\hat{f}-\chi\|$. Again, by using properties of Voronoi regions, it is possible to conclude that
\bes{
\|\hat{f}-\chi\|=\sqrt{\sum_{\omega\in\Omega}\int_{V^*_\omega}|\hat{f}(\hat{y})-\hat{f}(\omega)|^2\D \hat{y}}.
} 
In order to estimate $|\hat{f}(\hat{y})-\hat{f}(\omega)|^2$, for all $\omega\in\Omega$ and all $\hat y \in V^*_\omega$, Taylor's expansion of the entire function $\hat{f}$ is used. Therefore, by the Cauchy--Schwarz inequality we get
\ea{\label{step_ineq_1}
| \hat{f}(\hat{y}) - \hat{f}(\omega) |^2 &\leq \left ( \sum_{\alpha \neq 0} \frac{|(\hat{y}-\omega)^{\alpha} |}{\alpha!} |D^{\alpha} \hat{f}(\hat{y})| \right )^2 \nonumber
\\
& \leq \sum_{\alpha \neq0} \frac{c^{|\alpha|_{1}} (\hat{y}-\omega)^{2\alpha}  }{\alpha!} \sum_{\alpha \neq0} \frac{c^{-{|\alpha|_{1}}}}{\alpha!} | D^{\alpha} \hat{f}(\hat{y}) |^2,
}
for some constant $c>0$  to be determined later. The inequality \R{step_ineq_1} is where this proof starts to differ from Gr\"ochenig's original proof.
For the first term in \R{step_ineq_1}, by the multinomial formula \R{multinomial} we get
\eas{
\sum_{\alpha \neq0} \frac{c^{|\alpha|_{1}} (\hat{y}-\omega)^{2\alpha}  }{\alpha!} &= \sum^{\infty}_{k=0} \frac{c^k}{k!} \sum_{| \alpha |_{1} = k} \frac{k!}{\alpha!} (\hat{y}-\omega)^{2\alpha}  -1 \\
&= \sum^{\infty}_{k=0} \frac{c^{k}}{k!} |\hat{y}-\omega |^{2k} - 1 \\
&\leq \exp ( c (\delta_*c^*)^2) - 1,
}
where in the final inequality $\delta_*$-density of the set $\Omega$ is used:
\bes{
\forall \omega\in\Omega, \quad \forall \hat y\in V^*_\omega , \quad |\hat y -\omega|\leq \delta_* c^*.
}
Now consider the other term in \R{step_ineq_1}.  If we integrate over the Voronoi region $V^*_\omega$ and sum over $\omega\in\Omega$ then
\eas{
\sum_{\alpha \neq0} \frac{c^{-|\alpha|_{1}}}{\alpha!} \sum_{\omega\in\Omega}\int_{V^*_\omega} | D^{\alpha} \hat{f}(\hat{y}) |^2 \D \hat{y} &= \sum^{\infty}_{k=1} \frac{c^{-k}}{k!} \sum_{|\alpha|_{1} = k} \frac{k!}{\alpha!}\|D^{\alpha} \hat{f} \|^2
\\
&=\sum^{\infty}_{k=1} \frac{c^{-k}}{k!} \int_{E} \sum_{|\alpha|_{1} = k} \frac{k!}{\alpha!} (2\pi x)^{2\alpha}  |f(x) |^2 \D x, 
}
since by Parseval's identity
\bes{
\|D^{\alpha} \hat{f} \|^2 = \| \hat{F} \|^2 = \| F \|^2 = \int_{E} (2\pi x)^{2\alpha}|f(x)|^2\D x,
}
where $F(x)=(-\I2\pi x)^{\alpha}f(x)$.
Hence, again by the multinomial formula \R{multinomial}, we obtain
\eas{
\sum_{\alpha \neq0} \frac{c^{-|\alpha|_{1}}}{\alpha!} \sum_{\omega\in\Omega}\int_{V^*_\omega} | D^{\alpha} \hat{f}(\hat{y}) |^2 \D \hat{y} 
&= \sum^{\infty}_{k=1} \frac{c^{-k}(2 \pi m_E)^{2k}}{k!} \| f \|^2 \\ 
&= \left(\exp ( (2 \pi m_E)^2 / c ) -1 \right) \|f\|^2.
}
Therefore, from \R{step_ineq_1}, we get
\bes{
\| \hat{f} - \chi \|^2 \leq \left ( \exp(c (\delta_* c^*)^2) - 1 \right )\left(\exp ( (2 \pi m_E)^2 / c ) -1 \right) \| f \|^2.
}
If we equate the two terms, then we set $c = 2 \pi m_{E} / (\delta_* c^*)$ to get
\bes{
\| \hat{f} - \chi \| \leq  \left ( \exp(2 \pi m_E \delta_* c^*) - 1 \right ) \| f \|.
}
Thus \R{step_ineq_00} now gives
\bes{
\sqrt{B} \leq  \exp(2 \pi m_E \delta_* c^*) ,\quad \sqrt{A} \geq 2 - \exp(2 \pi m_E \delta_* c^*),
}
with the condition that
\bes{
\delta_* < \frac{\log2}{2 \pi m_{E} c^*},
}
as required.
}

To illustrate this result, let $E=\{x\in\bbR^d:|x|_{p}\leq1\}$, $p>0$, and let $\abs{\cdot}_*$ be the $\ell^q$ norm, $q\geq1$. Then, the density condition \R{delta_improved} becomes
\be{\label{delta_bound_norms}
\delta_{q}  < \frac{\log2}{2\pi \max\{1,d^{1/2-1/p}\}\max\{1,d^{1/2-1/q}\} },
}
due to \R{norms_ineq} and \R{me_norms}. This bound attains its minimum for $p=q=\infty$, when it deteriorates linearly with the dimension $d$. However, in all other cases the deterioration of the bound on density, and also, the deterioration of weighted frame bounds estimations, is slower with the dimension. Moreover, they are independent of dimension whenever $p\leq2$ and $q\leq2$. 

To compare this theorem with Gr\"ochenig's result given in Theorem \ref{t:Groch_improved}, we set $p=\infty$ and $q=2$ in \R{delta_bound_norms}.  The bound \R{delta_bound_norms} gives $\delta_2 < \frac{\log2}{2 \pi \sqrt{d}}$, whereas \R{Grochenig_delta} gives $\delta_2 < \frac{\log2}{2 \pi d}$.  Hence Theorem \ref{t:Groch_improved} leads to an improvement by a factor of $\sqrt{d}$ and no deterioration in the constant $\frac{\log2}{2 \pi}$.

\subsection{Sharp sufficient condition for weighted Fourier frames and the proof of Theorem \ref{t:weighted_frame}} \label{ss:sharp_result}

The relative separation of a sampling set  $\Omega$ is  necessary and sufficient for the existence of an upper frame bound \cite[Thm.\ 2.17]{Young}, see also \cite{Jaffard}.
However, if we introduce appropriate weights $\{\mu_\omega\}_{\omega\in\Omega}$ to compensate for the clustering of the sampling points $\Omega$, and consider $\{\sqrt{\mu_\omega} e_\omega\}_{\omega\in\Omega}$ instead of $\{ e_\omega \}_{\omega\in\Omega}$, then this condition ceases to be necessary, as it is evident from Gr\"ochenig's Theorem \ref{t:Grocg_orig} and the improved result given in Theorem \ref{t:Groch_improved}. On the other hand, the density condition from Theorem \ref{t:Groch_improved} that guarantees a lower weighted frame bound is still far from being sharp, while the sharp density condition from  Beurling's result, Theorem \ref{t:Beurling_thm}, does not guarantee a lower frame bound once nontrivial weights $\mu_\omega>0$ are introduced. To mitigate this, we next establish  Theorem  \ref{t:weighted_frame}.

Without imposing  restrictions such as separation, Theorem  \ref{t:weighted_frame} gives sufficient condition on a density of set of points to yield  a weighted Fourier frame, which is dimension independent. Therefore, in all dimensions, once this density condition is fulfilled, the sampling points are allowed to cluster arbitrarily, as long as the appropriate weights are used. Moreover, this result is sharp, which follows from the sharpness of Beurling's result, Theorem \ref{t:Beurling_thm}.

In order to prove Theorem \ref{t:weighted_frame}, we need the following lemma.

\begin{lemma}\label{l:seperated_sub}
If  $\Omega$ is a sequence with density  $\delta_{E^\circ}(\Omega)<1/4$ in $\hat{\mathbb{R}}^d$, then there exists a subsequence $\tilde \Omega \subseteq \Omega$ which is $\eta$-separated with respect to the norm $\abs{\cdot}_{E^\circ}$ for some $\eta>0$, and also has density  $\delta_{E^\circ}(\tilde \Omega)<1/4$ in $\hat{\mathbb{R}}^d$.
\end{lemma}

\begin{proof}To begin with, for the set $E$, we define $E(x,r)= x + r E$. For $\delta_{E^\circ}$, we simply write $\delta$.

Let us choose $\eta>0$ such that $\delta+\eta/2<1/4$ and set $\delta_1=\delta+\eta$. Now define $\tilde{\Omega}$ inductively as follows. For arbitrary picked point $\omega_0\in\Omega$, set $\tilde\omega_0=\omega_0$.
Given $\tilde{\omega}_0,\ldots,\tilde{\omega}_N$, define $\tilde{\omega}_{N+1}$ by
\bes{
\tilde \omega_{N+1} \in\Omega\cap E^\circ(x,\delta), 
}
where
\bes{
x\in\partial G=\partial\left(\bigcup_{\tilde\omega_n\in\tilde\Omega_N}E^\circ\left(\tilde\omega_n,\delta_1\right)\right)\quad \text{and} \quad \tilde\Omega_N=\{\tilde\omega_n\}_{n=0}^N.
}
Here, we picked any $x \in \partial G$ and then, for that $x$, any $\tilde{\omega}_{N+1} \in \Omega \cap E^\circ(x,\delta)$.
Finally, we let $\tilde{\Omega} = \{\tilde \omega_n \}^{\infty}_{n=0}$.

Note that for any $x \in \hat{\bbR}^d$ there must exists a point $\omega\in\Omega$ in the set $E^\circ(x,\delta)$ such that $x$ is covered by $E^\circ(\omega,\delta)$, since $\Omega$ is $\delta$-dense in the norm $\abs{\cdot}_{E^\circ}$ and $\hat{\bbR}^d$ can be covered by the sets $E^\circ(\omega,\delta)$, $\omega\in\Omega$. Moreover, for every $x\in \partial G$ a point $\omega \in\Omega\cap E^\circ(x,\delta)$ must be different than any other point $\omega\in\tilde{\Omega}_N$, since $\delta<\delta_1$.  Also, note that for every such $\omega\in  \Omega\cap E^\circ(x,\delta)$ it holds that
\bes{
\eta = \delta_1 - \delta \leq \inf_{\tilde\omega_n\in\tilde\Omega_N}|\omega-\tilde\omega_n|_{E^\circ} \leq \delta_1+\delta=2\delta+\eta.
}
Therefore if we choose $\tilde \omega_{N+1}$ from $\Omega\cap E^\circ(x,\delta)$ arbitrarily, and continue the procedure until $G=\hat{\bbR}^d$,
by the construction,  $\tilde\Omega$  is $\tilde\delta$-dense in the norm $\abs{\cdot}_{E^\circ}$ where $\tilde\delta=(2\delta+\eta)/2<1/4$. Moreover, it is $\eta$-separated with respect to the norm $\abs{\cdot}_{E^\circ}$.
\end{proof}

\rem{In view of this lemma, it might be tempting to infer the following
\be{\label{alternative}
\sum_{\omega\in\Omega} \mu_\omega|\hat{f}(\omega)|^2 \geq \sum_{\tilde\omega\in\tilde\Omega} \mu_{\tilde\omega}|\hat{f}(\tilde\omega)|^2 \geq \textnormal{meas}\left(\frac{\eta}{2}E^{\circ}\right) \sum_{\tilde\omega\in\tilde\Omega} |\hat{f}(\tilde\omega)|^2 \geq \textnormal{meas}\left(\frac{\eta}{2}E^{\circ}\right) A,
}
and therefore seemingly obtain the lower frame bound for the weighted non-separated sequence $\Omega$. However, note that the second inequality in \R{alternative} need not hold, since the weights at the very beginning are chosen as Lebesgue measure of the Voronoi regions corresponding to $\Omega$, which can be arbitrarily small due to clustering. Therefore, although the sequence $\tilde \Omega$ is separated, there might indeed exists $\tilde \omega \in \tilde \Omega $ such that its Voronoi region $V^{E^{\circ}}_{\tilde\omega}$ does not contain a ball of radius $\eta/2$ with respect to the $E^{\circ}$-norm. 
}

\prf{[Proof of Theorem $\ref{t:weighted_frame}$]
First of all, for the upper bound we use Theorem \ref{t:Groch_improved}. From the proof of Theorem \ref{t:Groch_improved}, we can infer that the density condition \R{delta_improved} is imposed only to ensure $A>0$, and that the estimate of the upper frame bound holds even if this density condition is not satisfied.   Indeed, for any compact set $E\subseteq\bbR^d$, any norm $\abs{\cdot}_*$ and any positive density $\delta_*<\infty$, the upper frame bound satisfies
\bes{
B \leq \exp{(4\pi m_E \delta_* c^*)} <  \infty.
}
In particular, if $\delta_{E^{\circ}}<1/4$, then
\bes{
B \leq \exp{(\pi m_E c^{\circ})} <  \infty,
}
where $c^{\circ}\in(0,\infty)$ is the smallest constant such that $\abs{\cdot}\leq c^{\circ} \abs{\cdot}_{E^{\circ}}$.

For the lower bound, we note that if $\Omega$ is separated, then everything follows easily. Namely, since $\Omega$ is $\eta$-separated with respect to the $E^{\circ}$-norm, we get 
\bes{
\sum_{\omega\in\Omega} \mu_\omega|\hat{f}(\omega)|^2 \geq \text{meas}\left(\frac{\eta}{2}E^{\circ}\right) \sum_{\omega\in\Omega} |\hat{f}(\omega)|^2 \geq \text{meas}\left(\frac{\eta}{2}E^{\circ}\right) A' \|f\|^2,
}
where $A'>0$ comes from application of Theorem \ref{t:Beurling_thm}. Thus we take $A=\text{meas}\left(\frac{\eta}{2}E^{\circ}\right) A'$.

However, if $\Omega$ is not separated, we proceed as follows. By Lemma \ref{l:seperated_sub}, we know that there exists a subsequence $\tilde \Omega  \subseteq \Omega$ with density $\delta_{E^\circ}(\tilde\Omega)=\delta_{E^\circ}(\Omega)+\eta/2<1/4$  and separation $\eta=\eta_{E^{\circ}}(\tilde\Omega)>0$.  Let $\epsilon < \eta/2$.  Then
\bes{
\sum_{\omega \in \Omega} \mu_\omega | \hat{f}(\omega) |^2 \geq \sum_{\tilde \omega \in \tilde \Omega} \sum_{\omega \in E^{\circ}_{\epsilon}(\tilde{\omega})\cap\Omega} \mu_\omega | \hat{f}(\omega) |^2,
}
where $E^{\circ}_{\epsilon}(\tilde{\omega})$ denotes the ball with respect to the $E^{\circ}$-norm of radius $\epsilon$ centered at $\tilde{\omega}$.
Since $\hat{f}$ is continuous function, from the Extreme value theorem, for each $\tilde \omega$, we know there is a point $z_{\tilde{\omega}} \in \overline{E^{\circ}_{\epsilon}(\tilde{\omega})} = E^{\circ}_{\epsilon}(\tilde{\omega})$, such that
\bes{
\forall \omega \in E^{\circ}_{\epsilon}(\tilde{\omega}), \quad | \hat{f}(\omega) | \geq | \hat{f}(z_{\tilde{\omega}}) |.
}
Since also $\mu_\omega=\text{meas}\left(V^{E^\circ}_\omega\right)$ and the sets $V^{E^\circ}_\omega$ are disjoint, we get
\bes{
\sum_{\omega \in \Omega} \mu_\omega | \hat{f}(\omega) |^2 
\geq \sum_{\tilde \omega \in \tilde \Omega} \left( | \hat{f}(z_{\tilde \omega}) |^2 \sum_{\omega \in E^{\circ}_{\epsilon}(\tilde{\omega})\cap\Omega} \mu_\omega\right) 
=  \sum_{\tilde \omega \in \tilde \Omega} \left(| \hat{f}(z_{\tilde{\omega}}) |^2 \text{meas}\left ( \bigcup_{\omega \in E^{\circ}_{\epsilon}(\tilde{\omega})\cap\Omega} V^{E^\circ}_\omega \right ) \right ).
}
Now we claim the following:
\bes{
\bigcup_{\omega \in E^{\circ}_{\epsilon}(\tilde{\omega})\cap\Omega} V^{E^\circ}_\omega  \supseteq E^{\circ}_{\rho}(\tilde{\omega}), \quad \rho=\frac{\epsilon}{2}.
}
To see this, let $|\hat y - \tilde{\omega} |_{E^{\circ}} \leq \frac{\epsilon}{2}$.  Since $\hat y \in V^{E^\circ}_\omega$ for some $\omega \in \Omega$, we have $|\hat y - \omega |_{E^\circ} \leq | \hat y - \tilde{\omega} |_{E^\circ}$. Therefore
\bes{
|\hat y - \omega |_{E^{\circ}} \leq| \hat y - \tilde{\omega} |_{E^{\circ}} \leq  \frac{\epsilon}{2},
}
and hence
\bes{
| \omega - \tilde{\omega} |_{E^{\circ}} \leq | \hat y - \omega |_{E^{\circ}} + | \hat y - \tilde{\omega} |_{E^{\circ}} \leq \epsilon.
}
Thus $\omega \in E^{\circ}_{\epsilon}(\tilde{\omega})\cap\Omega$ as required.  Therefore, we  get
\bes{
\sum_{\omega \in \Omega} \mu_\omega | \hat{f}(\omega) |^2 \geq \text{meas}\left(\frac{\epsilon}{2}E^{\circ}\right)\sum_{\bar \omega \in \bar \Omega} | \hat{f}(\bar \omega) |^2,
}
where $\bar \Omega= \{z_{\tilde\omega}\ :\ \tilde\omega \in \tilde\Omega\}$. To complete the proof, we only need to show that the set $\bar \Omega$ is separated and sufficiently dense, so that we can apply the Theorem \ref{t:Beurling_thm}.  Consider $\bar{\omega}_1$ and $\bar{\omega}_2$.  Then we clearly have
\bes{
| \bar{\omega}_1 - \bar{\omega}_2 |_{E^{\circ}} \geq \eta - 2 \epsilon  >0,
}
since $\tilde \Omega$ is separated with the separation $\eta$ and the $\bar{\omega}$'s lie in the $\epsilon$-cover of this set.  Moreover, it is straightforward to see that
\bes{
\delta_{E^\circ}(\bar \Omega) \leq \delta_{E^\circ} (\tilde \Omega) + \epsilon.
}
Thus, since $ \delta_{E^\circ} (\tilde \Omega)<1/4$,  we  have the same for $\bar \Omega$ for sufficiently small $\epsilon>0$. We set $A= \text{meas}\left(\frac{\epsilon}{2}E^{\circ}\right)A'$, where $A'>0$ is as in Theorem \ref{t:Beurling_thm} corresponding to sequence $\bar\Omega$, and finish the proof.
}

\begin{remark}\label{r:lower_bound}
\textnormal{From the proof of Theorem \ref{t:weighted_frame} and the proof of Lemma \ref{l:seperated_sub}, we can conclude the following. If $\Omega\subseteq\hat\bbR^d$ has density $\delta_{E^\circ}(\Omega)<1/4$, it yields a weighted Fourier frame with the lower weighted Fourier frame bound of the  form
\bes{
A= \text{meas}\left(\frac{\epsilon}{2}E^{\circ}\right)A', 
}
where $A'>0$ is the lower Fourier frame bound for  sequence $\bar{\Omega}\subseteq\hat\bbR^d$ with separation $\eta_{E^{\circ}}(\bar\Omega)=\eta-2\epsilon$  and density $\delta_{E^{\circ}}(\bar\Omega)\leq \delta_{E^{\circ}}(\Omega)+\eta/2+\epsilon$, where constants $\eta,\epsilon>0$ are such that $\epsilon < \eta/2$ and $\delta_{E^{\circ}}(\Omega)+\eta/2+\epsilon<1/4$. However, this does not in general lead to an explicit estimate of $A$ since we typically do not know an explicit estimate of $A'$. On the other hand, the upper weighted Fourier frame bound $B$ is explicitly estimated by 
$$B \leq \exp{(\pi m_E c^{\circ})},$$ 
where $c^{\circ}\in(0,\infty)$ is the smallest constant such that $\abs{\cdot}\leq c^{\circ} \abs{\cdot}_{E^{\circ}}$.} 
\end{remark}

\begin{remark}
Note that the density condition form Theorem \ref{t:Groch_improved} does not contradict the sharpness of the density condition from Theorem \ref{t:weighted_frame}, i.e.,~note that 
$$\frac{\log2}{2\pi m_E c^{\circ}}\leq\frac14,$$
where $c^{\circ}$ is the smallest constant such that $\abs{\cdot}\leq c^{\circ} \abs{\cdot}_{E^{\circ}}$ and $E$ is a compact, convex and symmetric set. To see this, we now argue that
$m_{E}c^{\circ}\geq1$. Note that from the definition of a polar set, it follows that for all $y\in\bbR^d$ we have
\bes{
\abs{y}_{E^{\circ}}=\max_{x\in E} \abs{x \cdot y}, 
}
see for example \cite{BenedettoSpiral}. Therefore $\abs{\cdot}_{E^{\circ}} \leq  m_{E} \abs{\cdot}$, which implies $1/m_E \leq c_{\circ}$, where  $c_{\circ}$ is the largest constant such that $c_{\circ} \abs{\cdot}_{E^{\circ}} \leq \abs{\cdot}$. Hence
\bes{
m_{E}c^{\circ}\geq \frac{c^{\circ}}{c_{\circ}},
}
and since $c_{\circ}\leq c^{\circ}$, the claim follows.
\end{remark}

To end this section, in order to illustrate differences between classical and weighted Fourier frames, as well as different uses of previously given results, let us consider the following two-dimensional example. 

\begin{example}\label{e:example_sphere}
\textnormal{Let $E=\cB_1\subseteq\bbR^2$  and let
\bes{
\Lambda_1 = \tfrac{1}{8} \bbZ^2, \qquad \Lambda_2=\left\{ \left(\frac1n,\frac1m\right)  :  (n,m)\in \bbZ^2,  \min\left\{|n|,|m|\right\}>8 \right\}.
}
Note that, for such $E$, $E^\circ=\cB_1$ and the $E^\circ$-norm is simply the Euclidean norm $\abs{\cdot}$.}

\textnormal{The set of points $\Lambda_1$ is separated with the density 
\bes{
\delta_{\cB_1}(\Lambda_1)=\frac{\sqrt{2}}{16}\approx 0.0884 < \frac{1}{4}.
}
Therefore, by Theorem $\ref{t:Beurling_thm}$, we conclude the family of functions $\{e_\lambda\}_{\lambda\in\Lambda_1}$ is a  frame for $\rL^2(\cB_1)$.   However, if we now consider the set 
\bes{
\Omega=\Lambda_1\cup\Lambda_2,
}
for which $\delta_{\cB_1}(\Omega) = \delta_{\cB_1}(\Lambda_1)= {\sqrt{2}}/{16}$, Theorem $\ref{t:Beurling_thm}$ can not be used since $\Omega$ has infinitely many accumulation points at 
\bes{
\{0\}\cup\left\{\left(\frac1n,0\right) : n\in\bbZ, |n|>8\right\}\cup\left\{\left(0,\frac1m\right): m\in\bbZ, |m|>8\right\},
}
and therefore it is not separated. Moreover, it can be verified that the family $\{e_\omega\}_{\omega\in\Omega}$ fails in satisfying the right inequality of $\R{frame}$. To see this, we first note that
\bes{
\int_{\cB_1} \E^{-2\pi\I \omega \cdot x} \D x = \frac{J_1(2\pi|\omega|)}{|\omega|}, 
}
where $J_1$ is the Bessel function of the first kind and order 1. Therefore, there exists $c>0$ such that
\be{\label{int_exp_rel}
c \leq \left| \int_{\cB_1} \E^{-2\pi\I\left(\frac1n x_1+\frac1m x_2\right)} \D x_1 \D x_2 \right|^2 \leq \pi^2,
}
for all $(n,m)\in \bbZ^2$ such that $\sqrt{{1}/{n^2}+{1}/{m^2}}< a {j'_{1,1}}/{(2\pi)}\approx0.6098$, where $a$ is some fixed constant from the interval $(0,1)$ 
and $j'_{1,1}$ is the first positive zero of the function $J_1$.  Hence, it is enough to take the function $g(x)={\bf1}_{\cB_1}(x)$ for which $\|g\|^2=\pi$, whereas $\sum_{\omega\in \Omega} | \hat{g}(\omega) |^2$ is unbounded. Thus, we conclude that the set $\Omega$ does not give a Fourier frame.}

\textnormal{On the other hand, if, for the same set of points $\Omega=\Lambda_1\cup\Lambda_2$, we consider the weighted family  $\{\sqrt{\mu_\omega}e_\omega\}_{\omega\in\Omega}$ with the weights defined as Voronoi regions in $\ell^2$-norm,  this particular function $g$ satisfies the relation \R{w_frame} with some $0<A,  B<\infty$. This can be easily proved by using the inequalities $\R{int_exp_rel}$, and the fact that 
\bes{
\sum_{n=9}^{\infty}\sum_{m=9}^{\infty}\left(\frac{1}{n-1}-\frac{1}{n+1}\right)\left(\frac{1}{m-1}-\frac{1}{m+1}\right) = \left(\frac{17}{72}\right)^2.
} 
which implies that the sum of Voronoi regions corresponding to the points $\Lambda_2$ converges.
Moreover, since $\delta_{\cB_1}(\Omega)=\sqrt{2}/16$, by Theorem  $\ref{t:weighted_frame}$ we conclude that $\Omega$ gives rise to a weighted Fourier frame.}

\textnormal{Note also, in order to verify that $\Omega$ forms a weighted Fourier frame,  Gr\"ochenig's original result  could not be used since 
\bes{
\delta_{\cB_1}(\Omega) = \frac{\sqrt{2}}{16}> \frac{\log2}{4\pi} \approx 0.0552. 
}
However, since in this case $m_E=1$ and $c^*=1$ and since
\bes{
\delta_{\cB_1}(\Omega) = \frac{\sqrt{2}}{16} <  \frac{\log2}{2\pi} \approx 0.1103,
} 
we are able to use Theorem $\ref{t:Groch_improved}$ to conclude that $\Omega$ generates a weighted Fourier frame with the weighted Fourier frame bounds
$\sqrt{A}\geq 0.2574$ and $\sqrt{B}\leq1.7426$. }
\end{example}

\section{Multidimensional function recovery} \label{s:NUGS}

Having provided guarantees for obtaining a weighted Fourier frame from a countable set of points, we now consider the question of function recovery from finite nonuniform Fourier data.  To do so, we shall use the generalized sampling approach for nonuniform samples (NUGS) from \cite{1DNUGS}. 
As in \cite{1DNUGS},  let $\Omega_N = \{\omega_n\}_{n=1}^N \subseteq \mathbb{\hat{R}}^d$ be a finite set of distinct frequencies, i.e.~the sampling scheme, let $\rT\subseteq\rH$ be a finite-dimensional subspace, the so-called \textit{reconstruction space}, and let $\{\hat{f}(\omega)\}_{\omega\in\Omega_N}$ be the given data of an unknown function $f\in\rH$. Under appropriate conditions, NUGS provides an approximation $\tilde f\in\rT$ to $f$ via the  mapping $F:f\mapsto\tilde{f}$, which depends only on the given data and which satisfies
\be{\label{NUGSstability}
\forall f, h \in \rH, \quad \| f - F(f+h) \| \leq C(\Omega_N,\rT) ( \| f - \cP_{\rT} f \| +\|h\| ),
}
for some constant $C(\Omega_N,\rT)>0$, where $\cP_{\rT}$ denotes the orthogonal projection onto $\rT$. Thereby, NUGS provides reconstruction $F(f)$, which is both \textit{quasi-optimal}, i.e.~close to the best approximation in the given reconstruction space  $\cP_{\rT}$, and \textit{stable}, i.e.~resistant to noisy measurements. In particular, the NUGS reconstruction is defined as 
\be{
\label{GS_LS_data}
\tilde{f} =  \underset{g \in \rT}{\operatorname{argmin}} \sum^{N}_{n=1} \mu_{\omega_n} \left | \hat{f}(\omega_n) - \hat{g}(\omega_n) \right |^2,
}
where $\mu_{\omega_n}>0$ are suitably chosen weights corresponding to the sampling points.

In what follows, by conveniently  using the results on weighted frames from the previous section, we prove that the NUGS reconstruction defined by \R{GS_LS_data} is stable and quasi-optimal---it satisfies \R{NUGSstability}---provided that the sampling scheme is sufficiently dense and wide in the frequency domain. By this, we shall extend guarantees of the NUGS framework from \cite{1DNUGS} to the multidimensional setting.


\begin{remark} 
\textnormal{Our purpose in this section is to provide analysis of recovery of a multivariate function $f$ from finitely many samples in an arbitrarily chosen subspace $\rT$ of finite dimension.  Consequently, we shall not address the specific algorithmic details, besides from noting that $\tilde{f}$ defined by \R{GS_LS_data} can be computed by solving an algebraic least squares problem. The computation of the NUGS reconstruction is summarized in \cite[Section 3.1]{1DNUGS}. For a general $\rT$, such that $\text{dim}(\rT)=M$, $\tilde{f}$ can be computed in $\ord{NM}$ operations. However,  if $\rT$ consists of $M$ wavelets, the computational complexity of NUGS can be reduced to only $\ord{N \log M}$ operations by using nonuniform fast Fourier transforms (NUFFTs) \cite{FesslerNUFFT,PottsNUFFT} and an iterative scheme for finding the least-squares solution such as the conjugate gradient method. This numerical implementation of NUGS is described at length in \cite{GataricPoon2015}.}
\end{remark}

Since we deal with finite sampling sets, which cannot be dense in the whole of $\mathbb{\hat{R}}^d$, in what follows we consider subsets of $\mathbb{\hat{R}}^d$.  Therefore, for a given sampling bandwidth $K>0$, we use the concept of \textit{$(K,\delta_*)$-density}:

\begin{definition}[$(K,\delta_*)$-density  with respect to $Y$]\label{K-dense}
Let $\Omega\subseteq \hat{\mathbb{R}}^d$ be a sampling scheme, $K>0$ and let  $\abs{\cdot}_*$ be an arbitrary norm on $\bbR^d$. Let $Y \subseteq\mathbb{\hat{R}}^d$ be a closed, simply connected set with $0$ in  its interior such that $\max_{\hat{y}\in Y}|\hat{y}|_{\infty}=1$. The set $\Omega$ is $(K,\delta_*)$-dense  with respect to $Y$ if
\begin{enumerate}
 \item[\textnormal{(i)}] $\Omega\subseteq Y_K$,  where $Y_K=KY$, and 
 \item[\textnormal{(ii)}] $\Omega$ is $\delta_*$-dense in the domain $Y_K$.
\end{enumerate}
\end{definition}

For a $K>0$ and a finite-dimensional space $\rT$,  let us define the \textit{$K$-residual} of $\rT$ as
\be{\label{res1}
R_K(\rT)=\sup_{\substack{f\in\rT \\ \|f\|=1}} \|\hat{f}\|_{\mathbb{\hat{R}}^d\backslash Y_K}.
}
Also, let $\Omega_N=\{\omega_n\}_{n=1}^{N}$ be  $(K,\delta_*)$-dense  with respect to $Y$, and $\Omega=\{\omega_n\}_{n\in\bbN}$, $\Omega_N\subseteq\Omega$, such that it yields a weighted Fourier frame. We make use of the following residual
\be{\label{res2}
\tilde{R}_K(\Omega_N,\rT)=\sup_{\substack{f\in\rT \\ \|f\|=1}}  \sqrt{\sum_{\omega \in \Omega \cap S_K }\mu_{\omega}|\hat{f}(\omega)|^2},
}
where $S_K=\hat{\bbR}^d \setminus E^\circ_{r(K)-1/2}$ and $E^\circ_{r(K)}$ is the largest inscribed ball with respect to $E^{\circ}$-norm inside $Y_K$.
Note that both of these residuals converge to zero when $K\rightarrow \infty$, since $\rT$ is finite-dimensional.   
We are ready to give our main result on NUGS.

\begin{theorem}
\label{t:admissible}
Let $\rT \subseteq \rH= \{ f \in \rL^2(\bbR^d) : \textnormal{supp}(f) \subseteq E \}$ be finite-dimensional, $E\subseteq\bbR^d$ compact, and let $\Omega_N=\{\omega_n\}_{n=1}^{N}$ be a sampling scheme.
\begin{enumerate}
  \item[\textnormal{I}] Let $\Omega_N$ be $(K,\delta_*)$-dense with respect to $Y$, with
\bes{
\delta_* < \frac{\log2}{2\pi m_E c^*},
}
where $\abs{\cdot}_*$ is an arbitrary norm on $\bbR^d$ and $c^*>0$ is the smallest constant such that $\abs{\cdot}\leq c^* \abs{\cdot}_*$.
Let also $\epsilon\in(0,\sqrt{\exp{(2\pi m_E \delta_* c^*)}\left(2-\exp{(2\pi m_E \delta_* c^*)}\right)})$.
If $K>0$ is large enough so that 
\bes{
R_K(T)\leq \epsilon,
}
then  the NUGS reconstruction $\tilde f$ given by $\R{GS_LS_data}$, with the weights defined as the measures of corresponding Voronoi regions with respect to $\abs{\cdot}_*$ in domain $Y_K$, exists uniquely and 
satisfies $\R{NUGSstability}$ with the reconstruction constant 
\be{\label{big_density}
C(\Omega_N,\rT) \leq \frac{2}{\sqrt{1-\epsilon^2 }   +1 - \exp{(2\pi m_E \delta_* c^*)} }.
} 
\item[\textnormal{II}] Let $E$ be also convex and symmetric, and $\Omega_N$ be $(K,\delta_{E^\circ})$-dense with respect to $Y$, with 
\bes{ 
\delta_{E^\circ}<\frac14.
}
Denote by $A$ the lower frame bound corresponding to the weighed Fourier frame arising from $\Omega=\{\omega_n\}_{n\in\bbN}$, $\Omega_N\subseteq\Omega$, and let $\epsilon\in(0,\sqrt{A})$.
If $K>0$ is large enough so that 
\bes{
\tilde{R}_K(\Omega_N,\rT)\leq \epsilon,
}
then  the NUGS reconstruction $\tilde f$ given by $\R{GS_LS_data}$, with the weights defined as the measures of corresponding Voronoi regions with respect to $\abs{\cdot}_{E^{\circ}}$ in  domain $Y_K$, exists uniquely and
satisfies $\R{NUGSstability}$ with the reconstruction constant 
\bes{
C(\Omega_N,\rT) \leq \frac{ \exp{\left(\frac12\pi m_{E} c^{\circ}\right)}}{\sqrt{A-\epsilon^2}},
}
where  $c^{\circ}>0$ is the smallest constant such that $\abs{\cdot}\leq c^{\circ} \abs{\cdot}_{E^{\circ}}$
\end{enumerate}
\end{theorem}

\begin{proof}
Let $\cS:\rH\rightarrow\rH$, $f\mapsto\cS f = \sum_{n=1}^{N}\mu_{\omega_n}\hat{f}(\omega_n) e_{\omega_n}$. By \cite[Thm.~3.3]{1DNUGS}, if there exist positive constants $C_1=C_1(\Omega_N,\rT)$ and $C_2=C_2(\Omega_N)$ such that
\be{\label{to_prove}
\forall f \in \rT,\quad\ip{\cS f}{f} \geq C_1 \| f \|^2, \qquad\qquad \forall f \in \rH, \quad\ip{\cS f}{f} \leq C_2 \| f \|^2,
}
then the NUGS reconstruction $\tilde f$ given by \R{GS_LS_data} exists uniquely and satisfies \R{NUGSstability} with 
\be{\label{rec_const}
C(\Omega_N,\rT)=\sqrt{C_2/C_1}.
}
Therefore, it is sufficient to prove \R{to_prove}. 

Now we define
\bes{
\chi(\hat y) = \sum_{\omega\in\Omega_N} \hat{f}(\omega){\bf1}_{V^*_\omega}(\hat{y}),\quad \hat{y} \in Y_K,
}
and observe  that $\|\chi\|^2_{Y_K} = \sum_{\omega\in\Omega_N} \mu_{\omega}|\hat{f}(\omega)|^2$.
Note also
\bes{
\|f\|_{Y_K} - \|\hat{f}-\chi\|_{Y_K} \leq \|\chi\|_{Y_K} \leq \|\hat{f}-\chi\|_{Y_K} + \|f\|,
}
and, by the same reasoning as in the proof of  Theorem \ref{t:Groch_improved}, we obtain
\bes{
\|\hat{f}-\chi\|_{Y_K} \leq  \left ( \exp(2 \pi m_E \delta_* c^*) - 1 \right ) \| f \|.
}
Therefore for all $f \in \rH \backslash \{0\}$
\be{\label{step}
\left( \sqrt{1-\frac{\|\hat f\|^2_{\mathbb{\hat{R}}^d\backslash Y_K} }{\|f\|^2} }   +1 - \exp{(2\pi m_E \delta_* c^*)} \right)^2 \| f \|^2 
\leq \sum_{\omega\in\Omega_N} \mu_{\omega}|\hat{f}(\omega)|^2 \leq \exp(4 \pi m_E \delta_* c^*) \| f \|^2.
}
Hence, if $\delta_* < \log2/(2\pi m_E c^*)$ we have $\sqrt{C_2} \leq 2$ and 
\bes{
\sqrt{C_1} \geq  \sqrt{1- \epsilon^2}   +1 - \exp{(2\pi m_E \delta_* c^*)>0},
}
due to the definition of $R_K(\rT)$ \R{res1} and the assumption that 
\bes{
R_K(\rT)\leq\epsilon<\sqrt{\exp{(2\pi m_E \delta_* c^*)}\left(2-\exp{(2\pi m_E \delta_* c^*)}\right)}.
}
The first statement follows directly by using \R{rec_const}. For the second statement, where $\delta_{E^{\circ}}<1/4$, due to \R{step}, we have  $\sqrt{C_2} \leq \exp{(\pi m_{E} c^{\circ}/2)}$. However, for the lower bound  $C_1$ we proceed as follows by using Theorem \ref{t:weighted_frame}.
Since Voronoi regions are taken with respect to $Y_K$ instead of $\hat{\bbR}^d$, we need a subsequence $\Omega_N'\subseteq\Omega_N$ which has points sufficiently far from $\partial Y_K$ so there is no any change in Voronoi regions. Since $\delta_{E^\circ}<1/4$, we can take $\Omega_N'\subseteq E^\circ_{r(K)-1/2}$, where $E^\circ_{r(K)}$ is the largest inscribed ball with respect to $E^{\circ}$-norm inside $Y_K$.  Note that 
$$\Omega\setminus\Omega_N' \subseteq \Omega \cap  \left( \hat{\bbR}^d \setminus E^\circ_{r(K)-1/2}\right).$$ 
Denote $S_K=\hat{\bbR}^d \setminus E^\circ_{r(K)-1/2}$. Therefore
\eas{
\sum_{\omega\in\Omega_N} \mu_{\omega} \abs{\hat{f}(\omega)}^2 &\geq \sum_{\omega\in\Omega} \mu_{\omega} \abs{\hat{f}(\omega)}^2 - \sum_{\omega\in\Omega \setminus \Omega_N'} \mu_{\omega} \abs{\hat{f}(\omega)}^2 \\
& \geq A \|f\|^2 - \sum_{\omega\in\Omega \cap S_K} \mu_{\omega} \abs{\hat{f}(\omega)}^2.
}
where the existence of $A>0$ is provided by  Theorem \ref{t:weighted_frame}. Hence, by \R{res2},  for $C_1$ we have
\bes{
C_1 \geq A - \tilde{R}_K(\Omega_N,\rT)^2 \geq A - \epsilon^2>0.
}
Now the result follows due to \R{rec_const}.
\end{proof}

By this theorem, for a fixed  reconstruction space $\rT$,  a stable and quasi-optimal multivariate reconstruction via NUGS is guaranteed subject to sufficiently large sampling bandwidth $K$ and exactly the same sampling densities derived in Theorems \ref{t:weighted_frame} and \ref{t:Groch_improved}  that were shown to guarantee a weighted Fourier frame. In particular, in part II of this theorem, we do not require sampling density to increase in higher dimensions. However, since the lower frame bound $A$ in general is not known, this part does not provide explicit bound on the reconstruction constant $C(\Omega,\rT)$ that indicates stability and accuracy of the reconstruction. As an alternative, one can use part I of the theorem which does provide explicit bound but under more stringent  density condition. 

Additionally, residual $\tilde R_K(\Omega_N,\rT)$ used in part II of this theorem depends on both $\rT$ and $\Omega_N$, while residual $R_K(\rT)$ used in part I depends only on $\rT$.  Thus, by explicit bound \R{big_density} of the first part of this theorem, we are able to largely separate the geometric properties of the sampling scheme, i.e.~the density, from intrinsic properties of the reconstruction space $\rT$, i.e.~the $K$-residual $R_K(\rT)$.  The latter is determined solely by the decay of functions $\hat{f}$, $f \in \rT$, outside the domain $Y_K$. In other words, once $R_K(\rT)$ is estimated for any given subspace $\rT$ (see \S \ref{s:conclusions} for a discussion on this point), we can ensure a stable and quasi-optimal reconstruction for \textit{any}  nonuniform sampling scheme which is $(K,\delta_*)$-dense with small enough $\delta_*$.

\subsection{Relation to previous work} \label{ss:implementation_and_relation}

The function recovery method NUGS used in this paper is based on the work of the authors \cite{1DNUGS}. This is a special instance of a more general approach of sampling and reconstruction in abstract Hilbert spaces, known as generalized sampling (GS).  Although introduced by two of the authors in \cite{BAACHShannon} it has its origins in earlier work of Unser \& Aldroubi \cite{unser1994general}, Eldar \cite{eldar2003sampling}, Eldar \& Werther \cite{eldar2005general}, Gr\"ochenig \cite{GrochenigTrigonometric, GrochenigModernSamplingBook}, Hrycak \& Gr\"ochenig \cite{hrycakIPRM}, Shizgal \& Jung \cite{PolynomialJung},
Aldroubi \cite{AldroubiAverageSamp} and others.

In  \cite{GrochenigTrigonometric} (see also \cite{GrochenigModernSamplingBook, GrochenigStrohmerMarvasti,  FeichtingerEtAlEfficientNonuniform}), the problem of recovering a bandlimited function from its own nonuniform samples was considered, where the arbitrary clustering is addressed by using weighted Fourier frames, exactly the same as we do in this paper.  Specifically, Gr\"ochenig et al.~developed an efficient algorithm for the nonuniform sampling problem, known as the ACT algorithm (Adaptive weights, Conjugate gradients, Toeplitz) where they consider the reconstruction of bandlimited functions in a particular finite-dimensional space consisting of trigonometric polynomials.  This corresponds to a specific instance of NUGS with a Dirac basis for $\rT$. The recovery model of compactly supported functions in a Dirac basis, with applications to MRI, was considered in \cite{KnoppKunisPotts}. As discussed in \cite{1DNUGS}, the main advantage offered by NUGS is that it allows for arbitrary reconstruction subspaces $\rT$. For example, $\rT$ may consist of compactly supported wavelets since it is well-known that multidimensional images in applications such as MRI and CT are well represented using compactly supported wavelets \cite{UnserAldroubiWaveletReview}.

The result from Theorem \ref{t:admissible}, extends the work of Gr\"ochenig et al.~in two ways.  First, we have a less stringent density requirement based on the bounds derived in Theorems \ref{t:weighted_frame} and \ref{t:Groch_improved}.  Second, we allow for arbitrary choices of $\rT$ which can be tailored to the particular function $f$ to be recovered. In particular, convergence and stability of the ACT algorithm  \cite[Thm.\ 7.1]{GrochenigModernSamplingBook} are guaranteed by the sufficient sampling density and the explicit weighted frame bounds given in \cite[Prop.\ 7.3]{GrochenigModernSamplingBook} (Theorem \ref{t:Grocg_orig} here). Therefore, the bounds derived in Theorem \ref{t:Groch_improved} directly improve the guarantees for ACT algorithm. Moreover, the bounds derived in Theorem \ref{t:Groch_improved} directly improve the existing estimates from \cite{KunisPottsSIAM,PottsTasche} for efficient and reliable computation of trigonometric polynomials, which are based on Gr\"ochenig's original bounds from \cite{GrochenigModernSamplingBook}.

On the other hand, in MRI and several other applications, a popular algorithm for reconstruction from nonuniform Fourier samples is known as the iterative reconstruction techniques \cite{FesslerFastIterativeMRI}, see also \cite{pruessmann2001}.  This can also be viewed as an instance of NUGS, where $\rT$ is a space of piecewise constant functions on a $M \times M$ grid (the term `iterative' refers to the use of conjugate gradients to compute the reconstruction).  Equivalently, when $M$ is a power of $2$, then $\rT$ can be expressed as the space spanned by Haar wavelets up to some finite scale.  As a result, Theorem \ref{t:admissible} also provides guarantees for the iterative reconstruction techniques.  Importantly, we shall also show how NUGS allows one to obtain better reconstructions, by replacing the Haar wavelet choice for the subspace $\rT$ with higher-order wavelets.

In addition to aforementioned algorithms, it is also worth mentioning that there exists a vast wealth of other methods for solving the same (or equivalent) recovery problem from nonuniform Fourier samples that are fundamentally different than ours. Unlike some common approaches in MRI, such as gridding \cite{JacksonEtAlGridding}, resampling \cite{RosenfeldURS2} or earlier mentioned iterative algorithms \cite{FesslerFastIterativeMRI}, we do not model  $f$ as a finite-length Fourier series, or as a finite array of pixels, but rather as a function in $\rL^2$-space. Hence, by using an appropriate approximation basis, we successfully avoid the unpleasant artefacts (e.g. Gibbs ringing) associated with gridding and resampling algorithms and also we gain more accuracy than with the iterative algorithms (see \S \ref{s:num_results}). On the other hand, there are approaches commonly found in nonuniform sampling theory which do use analog model but whose reconstruction is based on an iterative inversion of the frame operator \cite{BenedettoFrames, BenedettoSpiral,  FeichtingerGrochenigIrregular, AldroubiGrochenigSIREV}. Since in practice one has only finite data, these approaches typically lead to large truncation errors (similar to Gibbs phenomena), and additionally, a long computational time in more than one dimension.

\section{Examples of sufficiently dense sampling schemes} \label{s:dense_ss}

In the next section, we illustrate NUGS  on several numerical examples, where we use a number of sampling schemes commonly found in practice. Herein, we consider functions supported on $E=[-1,1]^2$. In Theorem \ref{t:weighted_frame}, we require a sampling scheme $\Omega$ to satisfy 
\be{\label{delta_14}
\delta_{E^\circ}(\Omega) < \frac{1}{4},
}
where $E^\circ$ is the unit ball in $\ell^1$-norm, or, according to Theorem \ref{t:Groch_improved}, a more strict density condition
\be{\label{delta_small}
\delta_{\cB_1}(\Omega) < \frac{\log2}{2\pi m_E}
}
(we have chosen  $\abs{\cdot}_*=\abs{\cdot}$ for simplicity). Recall that $m_{E}=\sqrt{2}$ if $E=[-1,1]^2$.
In this section, we construct some sampling schemes such that they satisfy these density conditions.
Note that for $E=[-1,1]^2$ we have 
\bes{
\delta_{E^\circ} (\Omega) \leq \sqrt{2} \delta_{\cB_1} (\Omega). 
}
Hence, to have \R{delta_14} it is enough to enforce $\delta_{\cB_1}(\Omega) < {1}/{(4\sqrt{2})}$. The condition 
\be{\label{delta_D}
\delta_{\cB_1}(\Omega) < D, 
}
where $D>0$ is a given constant,  can be easily checked on a computer for an arbitrary nonuniform sampling scheme $\Omega$. Moreover, as we shall show below, for special sampling schemes, e.g. polar and spiral, it is always possible to construct them so that they satisfy the condition \R{delta_D}. The advantage of considering density condition in the Euclidean norm lies in its symmetry.

We mention that in \cite{BenedettoSpiral}, one can find a construction of a spiral sampling scheme satisfying condition \R{delta_D}. Here, we use a slightly different spiral scheme, one which has an accumulation point at the origin and cannot be treated without weights. More precisely, we use the constant angular velocity spiral, whereas Benedetto \& Wu  \cite{BenedettoSpiral} use the constant linear velocity spiral (see \cite[Fig 2]{spyralScienceDirect}). Also, besides giving a sufficient condition for a spiral sampling scheme in order to satisfy \R{delta_D}, we provide both sufficient and necessary condition such that polar and jittered sampling schemes are appropriately dense.

\subsection{Jittered sampling scheme}

This sampling scheme is a standard model for jitter error, which appears when the measurement device is not scanning exactly on a uniform grid; see Figure \ref{fig:ss}. Due to its simplicity, we can consider directly the condition \R{delta_14}, and then, for completeness, we consider also \R{delta_small}.
For a given sampling bandwidth $K>0$ and parameters $\epsilon>0$ and $\eta\geq0$, we define the jittered sampling scheme as
\be{\label{jitter_ss}
\Omega_K=\left\{(n,m)\epsilon+\eta_{n,m}: n,m=-\lfloor K/\epsilon \rfloor,\ldots,\lfloor K/\epsilon \rfloor\right\},
}
where $\eta_{n,m}=(\eta^{x\vphantom{y}}_{n,m},\eta^{\vphantom{x}y}_{n,m})$ with $\eta^{x\vphantom{y}}_{n,m}$ and $\eta^{\vphantom{x}y}_{n,m}$ such that  $|\eta^x_{n,m}|,|\eta^y_{n,m} | \leq\eta$. Note that $\Omega_K\subseteq Y_{K'}=K'[-1,1]^2$, where $K'=\epsilon\lfloor K/\epsilon \rfloor+\eta$.
Now, the following can easily be seen:

\prop{Let $E=[-1,1]^2$. Let also $K>0$, $\epsilon>0$ and $\eta\geq0$ be given, and define $K'=\epsilon\lfloor K/\epsilon \rfloor+\eta$. The sampling scheme $\Omega_K$ defined by $\R{jitter_ss}$ is 
\begin{enumerate}
 \item $(\delta_{E^\circ},K')$-dense with respect to $Y=[-1,1]^2$ and with $\delta_{E^\circ}<1/4$  if and only if 
$\epsilon+2\eta<1/4$.
 \item $(\delta_{\cB_1},K')$-dense with respect to $Y=[-1,1]^2$ and with $\delta_{\cB_1}<(\log2)/(2\pi\sqrt{2})$  if and only if $\epsilon+2\eta<(\log2)/(2\pi)$.
\end{enumerate}
}

\subsection{Polar sampling scheme}\label{ss:polar}

Here, we discuss an important type of sampling scheme used in MRI and also whenever the Radon transform is involved in sampling process, see Figure \ref{fig:ss}.
For a given sampling bandwidth $K>0$ and separation  between consecutive concentric circles $r>0$ we define a polar sampling scheme as 
\be{\label{polar_ss}
\Omega_K=\left\{m r \E^{\I n\Delta \theta}\ :\  m=-\lfloor K/r \rfloor,\ldots,\lfloor K/r \rfloor,\ n=0,\ldots, N-1  \right\},
}
where $\Delta \theta = \pi/N \in(0,\pi)$ is the angle between neighbouring radial lines and $N\in\mathbb{N}$ is the number of radial lines in the upper half-plane. Note that $\Omega_K \subseteq \cB_{r\lfloor K/r \rfloor}\subseteq \hat\bbR^2$.  In what follows we shall assume that $K/r\in\bbN$ for simplicity.

\prop{Let  $D>0$, $K>D$, and $ r\in(0,2D)$ be given such that $K/r\in\bbN$.
The sampling scheme $\Omega_K$ defined by $\R{polar_ss}$ is $(K,\delta_{\cB_1})$-dense with respect to $Y=\cB_1$ and with
\bes{
\delta_{\cB_1} (\Omega_K) < D
}
if and only if
\be{\label{theta_polar_con}
\Delta \theta < 2 \min\left\{\arctan \frac{\sqrt{ D^2-\left({r}/{2}\right)^2 }}{K-r/2} , \arccos\left(1-\frac{D^2}{2K^2}\right) \right\}.
}
}

\begin{proof}
To prove this claim, we need to calculate
\bes{
\delta_{\cB_1} (\Omega_K)=\sup_{\hat{y}\in\cB_K}\inf_{\omega\in\Omega_K}|\hat{y}-\omega|_{\cB_1}.
}
First note that, due to the definition of Voronoi regions \ref{Voronoi}, we have
\be{\label{delta_polar_ex}
\delta_{\cB_1}(\Omega_K)=\sup_{\omega\in\Omega_K}\sup_{\hat{y}\in V_{\omega}}|\hat{y}-\omega|_{\cB_1},
} 
where $V_{\omega}$ is the Voronoi region at $\omega$ with respect to the Euclidean norm and inside the domain $\cB_K$. Therefore, we have to find the maximum radius of all Voronoi regions inside $\cB_K$, where the radius of a Voronoi region $V_{\omega}$ is defined as the radius of the Euclidean ball described around $V_{\omega}$ and  centered at $\omega$. Since the Voronoi regions are taken with respect to the Euclidean norm, they are convex polygons \cite{Klein}, and hence, the Voronoi radius is always achieved at a vertex which is furthest away from the center.

Since $\Omega_K$ is a polar sampling scheme with the uniform separation between consecutive concentric circles, the largest Voronoi radius is achieved at some of the vertices positioned between the two most outer circles of $\Omega_K$, including the most outer circle. Note that, by the definition of Voronoi regions, a joint vertex of two adjacent Voronoi regions $V_{\omega}$ and $V_{\omega'}$ is equally distant from both points $\omega$ and $\omega'$.  Therefore, without loss of generality, in \R{delta_polar_ex}, we may consider only the sampling points form $\Omega_K$ that are at the most outer circle.

Next, since $\cB_1$ is symmetric with respect to any direction, and due to the symmetry of a polar sampling scheme, in \R{delta_polar_ex}, without loss of generality we may assume, that $\omega = K\E^{\I 0}$, and $\hat y \in  \left\{s\E^{\I \theta}: s\in(K-r,K],\ \theta\in[0,\Delta\theta/2]\right\} \cap V_\omega$. Denote $\omega'=(K-r)\E^{\I 0}$.
We now conclude that \R{delta_polar_ex} is achieved at some of the following two vertices of $V_{\omega}$, which are also the only vertices of $V_\omega$ contained in the region $\left\{s\E^{\I \theta}: s\in(K-r,K],\ \theta\in[0,\Delta\theta/2]\right\}$:
\begin{enumerate}
 \item $v_1= \frac{K-r/2}{\cos\theta_0} \E^{\I\Delta\theta/2}$, which is the joint vertex for adjacent $V_{\omega}$ and $V_{\omega'}$ lying on the radial line corresponding to angle $\Delta\theta/2$, at the equal distance $d_1(\Delta\theta)$ from both points $\omega$ and $\omega'$. This point $v_1$ is easily calculated by equating the distances $|s\E^{\I\Delta\theta/2}-\omega|$ and $|s\E^{\I\Delta\theta/2}-\omega'|$. One derives
 $$d_1(\Delta\theta)=\sqrt{\left({r}/{2}\right)^2+\left(\left(K-{r}/{2}\right)\tan(\Delta\theta/2)\right)^2}.$$
 \item $v_2=K\E^{\I\Delta\theta/2}$, which is a vertex of $V_{\omega}$ lying on the radial line corresponding to $\Delta\theta/2$ and at the most outer circle, at the distance 
 $$d_2(\Delta\theta)=K\sqrt{2-2\cos(\Delta\theta/2)}.$$
\end{enumerate}
Hence, having $\delta_{\cB_1} (\Omega_K) < D$ in the domain $\cB_K$ is equivalent to 
\bes{
\max\{d_1(\Delta\theta),d_2(\Delta\theta)\} < D.
}
This is equivalent to
\bes{
\Delta \theta < 2 \min\left\{\arctan \frac{\sqrt{ D^2-\left({r}/{2}\right)^2 }}{K-r/2} , \arccos\left(1-\frac{D^2}{2K^2}\right) \right\},
}
which proves our claim.
\end{proof}

This proposition asserts that $\delta_{\cB_1}$-density of a polar sampling scheme is satisfied if and only if the corresponding angle $\Delta\theta$ is sufficiently small and taken according to the formula \R{theta_polar_con}. From \R{theta_polar_con}, it is evident that the angle $\Delta\theta$ goes to zero when $K\rightarrow \infty$. Therefore, the condition $\delta_{\cB_1} (\Omega_K) < D$ implies that the points $\Omega_K$ accumulate at the inner concentric circles as $K$ increases.  Thus, the unweighted frame bounds for the frame sequence corresponding to $\Omega_K$ clearly blow up as $K \rightarrow \infty$, which can be prevented by using the weights.

\subsection{Spiral sampling scheme}\label{ss:spiral}
For a given $r>0$,  
\be{\label{spiral_traj}
S_r(\theta)=r\tfrac{\theta}{2\pi}\E^{\I  \theta}, \quad \theta\geq0,
}
is a spiral trajectory in $\hat{\bbR}^2$ with the constant separation $r$ between the spiral turns. If $\theta\in[0,2\pi k]$ for $k\in\bbN$, then the number of turns in the spiral is exactly $k$. For given $r>0$ and $k\in\mathbb{N}$, let $Y_{rk}\subseteq \mathbb{\hat{R}}^2$ be defined as 
\be{\label{spiral_domainY}
Y_{rk}= \left\{  S_{\rho}(\theta) : \rho\in[0,r],\ \theta \in[0,2 \pi k] \right\},
}
Then  $S_r(\theta)\subseteq Y_{rk} \subseteq \cB_{rk}$, for $\theta\in[0,2\pi k]$. 

Now, let $K>0$ and $r>0$ be given, and for simplicity assume that they are such that $K/r=k\in\bbN$. We define a spiral sampling scheme as 
\be{\label{spiral_ss}
\Omega_K=\left\{r \tfrac{n\Delta\theta}{2\pi}\E^{\I  n\Delta\theta} : n=0,\ldots, Nk \right\}.
}
where $\Delta \theta= 2\pi/N \in(0,\pi)$, $N\in\bbN$, is a discretization angle.
Note that this $\Omega_K$ represents a discretization of the spiral trajectory \R{spiral_traj}, which  consists of $k$ turns with the constant separation $r$ between them and with a constant angular distance $\Delta \theta$. Also, note that $\Omega_K\subseteq Y_K = KY\subseteq \cB_K \subseteq \hat{\bbR}^2$, where $Y$ is 
\be{\label{spiral_domainY1}
Y= \left\{  \rho\tfrac{\theta}{2\pi}\E^{\I  \theta} : \rho\in[0,1],\ \theta \in[0,2 \pi] \right\}
}
i.e., $Y$ is given by \R{spiral_domainY} for $r=k=1$.

\prop{Let $D>0$, $K>4/5 D$ and let $r\in(0,2D)$  be given such that $K/r=k\in\bbN$. The sampling scheme $\Omega_K$ defined as $\R{spiral_ss}$ is  $(\delta_{\cB_1},K)$-dense with respect to $Y$ given by $\R{spiral_domainY1}$ and with 
\bes{
\delta_{\cB_1}(\Omega_K)<D
} 
if  
the angle $\Delta \theta$ is chosen small enough depending on $k$.
}

\begin{proof}
To prove this claim, we want to estimate $\delta_{\cB_1}(\Omega_K)$. First note that the distance from any point inside region $Y_{rk}$ to the spiral trajectory $S_r(\theta)$, $\theta\in[0,2\pi k]$, is at most $r/2$, see  \cite[Eq.~(18)]{BenedettoSpiral}. Also, note that the distance from any point on the spiral trajectory $S_r(\theta)$, $\theta\in[0,2\pi k]$, to a point from $\Omega_K$ is at most $|S_r(2\pi k)-S_r(2\pi k-\Delta\theta/2)|$. Hence,  as in \cite{BenedettoSpiral}, by the triangle inequality we obtain
\bes{
\delta_{\cB_1}(\Omega_K)\leq\frac{r}{2}+|S_r(2\pi k)-S_r(2\pi k-\Delta\theta/2)|.
}
Therefore, the density condition is satisfied if $\Delta\theta$ is such that
\bes{
d_{r,k}(\Delta\theta)=|S_r(2\pi k)-S_r(2\pi k-\Delta\theta/2)| < D-\frac{r}{2}.
}
Hence, it is enough to choose $\Delta\theta$ as
\bes{
\Delta\theta < \tilde\theta,
}
where $\tilde\theta$ is such that $d_{r,k}(\tilde\theta)=D-r/2$. This $\tilde \theta$ exists and it is unique on the interval $(0,\pi)$, since the function $d_{r,k}(\cdot)$ is continuous and strictly increasing on $(0,\pi)$ and also
\bes{
\lim_{\Delta\theta\rightarrow0}d_{r,k}(\Delta\theta)= 0 < D-\frac{r}{2}, \quad  \lim_{\Delta\theta\rightarrow\pi}d_{r,k}(\Delta\theta)=  r \sqrt{k^2+\left(k-\frac14\right)^2}  \geq  \frac54 K > D-\frac{r}{2}.
}
\end{proof}

Let us mention here that in a similar manner an interleaving spiral sampling scheme can be analyzed. An interleaving spiral  consists of multiple single spirals. Both of these spiral sampling schemes are shown in Figure \ref{fig:ss}.

\section{Numerical results}\label{s:num_results}

Finally, in this section, we present several numerical experiments illustrating some of the developed theory.

First, we demonstrate the use of weights when reconstructing from nonuniform Fourier measurements. Some of the advantages of using  weights have been already  reported earlier in the literature, see for example \cite{FeichtingerGrochenigIrregular,FeichtingerEtAlEfficientNonuniform,GrochenigStrohmerMarvasti} and also \cite{JacksonEtAlGridding,FesslerFastIterativeMRI}. In a different setting, in Figure \ref{fig:weights_new}, we provide further insight on  the necessity of using weights. To this end, we test a polar sampling scheme which is constructed as in \S \ref{ss:polar}. From the given set of samples, we perform function recovery using NUGS with boundary corrected Daubechies wavelets of order 1, 2 and 3, as well as the direct recovery approach called gridding \cite{JacksonEtAlGridding}. We perform function recovery with and without using weights, using 10 iterations in the conjugate gradient method used for solving the least squares corresponding to the NUGS reconstruction \R{GS_LS_data}. As shown in Figure \ref{fig:weights_new}, the reconstruction error without using weights does not exceed order $10^{-2}$. Hence, the advantages of higher order wavelets cannot be easily exploited in this case, as opposed to the case when reconstructing with weights. Moreover, the gridding reconstruction obtained without using weights is distinctly inferior. We recall that gridding reconstruction is computed with only one iteration, i.e.~with a single  NUFFT.

 \begin{figure}
\centering
$\begin{array}{ccccc}
 & \text{\small{Haar}} & \text{\small{DB2}} & \text{\small{DB3}} & \text{\small{gridding}} \\
{\rotatebox{90}{\qquad\quad\text{\small{weights}}}} & \includegraphics[scale=0.238, trim=0cm 0cm 0cm 0cm]{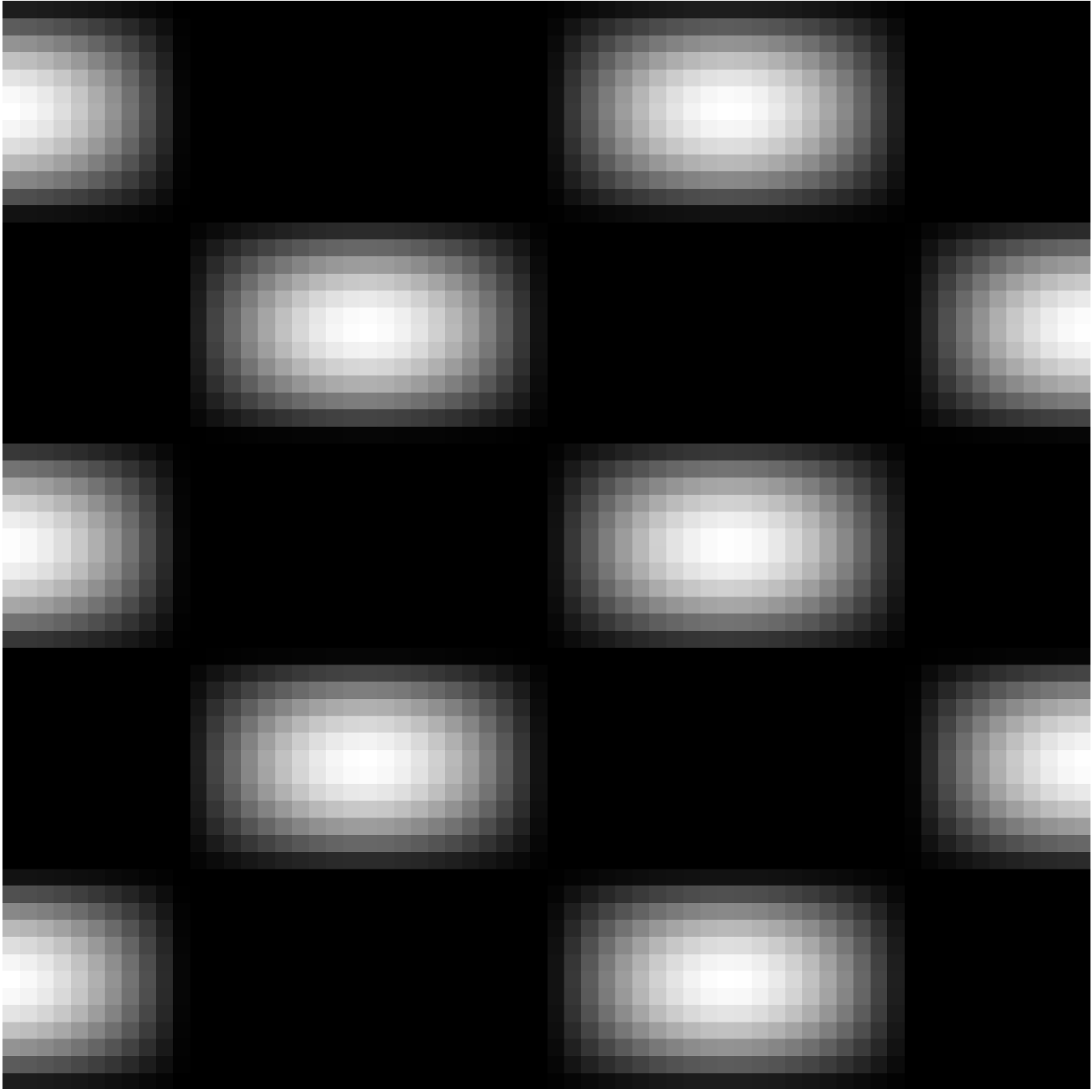} & \includegraphics[scale=0.238, trim=0cm 0cm 0cm 0cm]{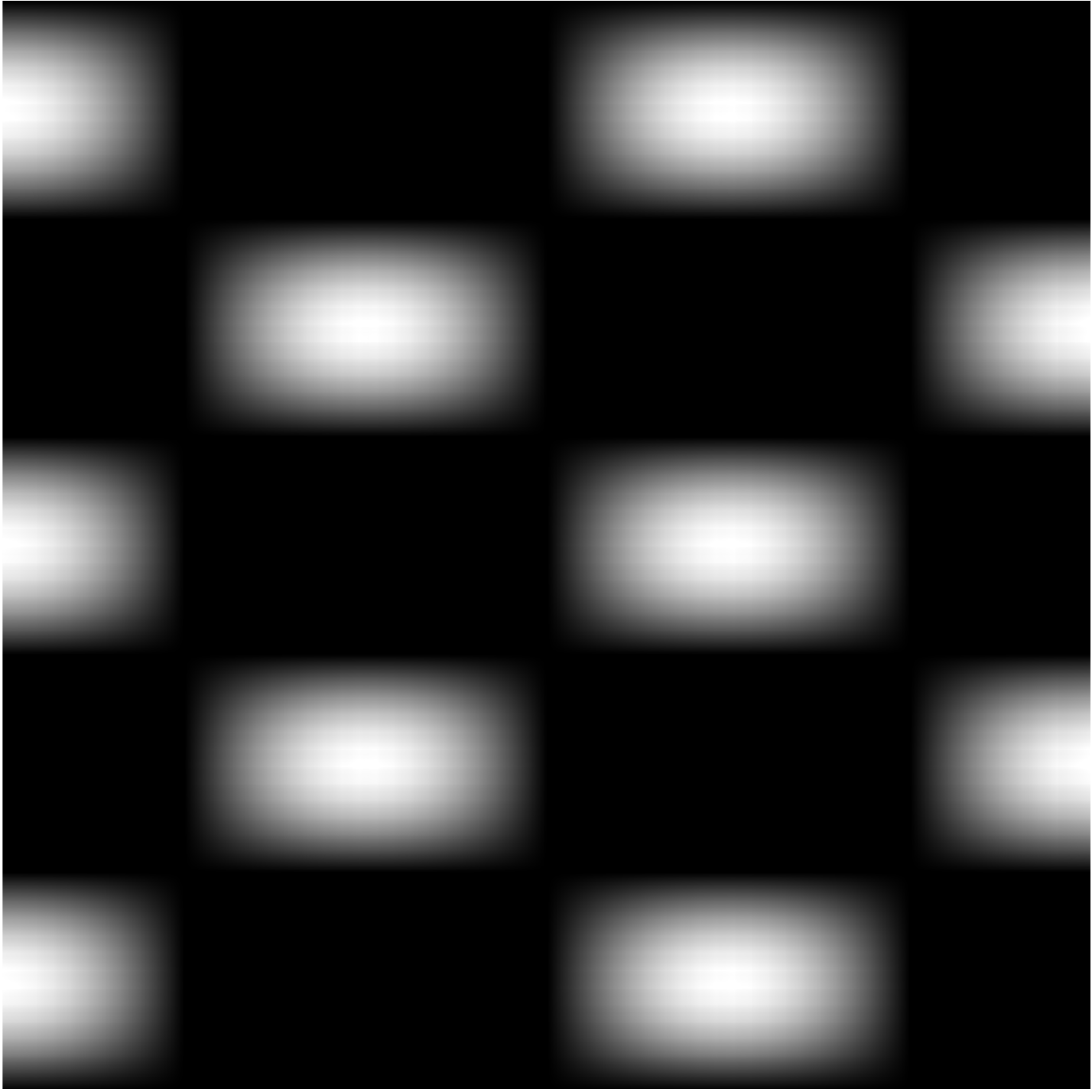} & \includegraphics[scale=0.238, trim=0cm 0cm 0cm 0cm]{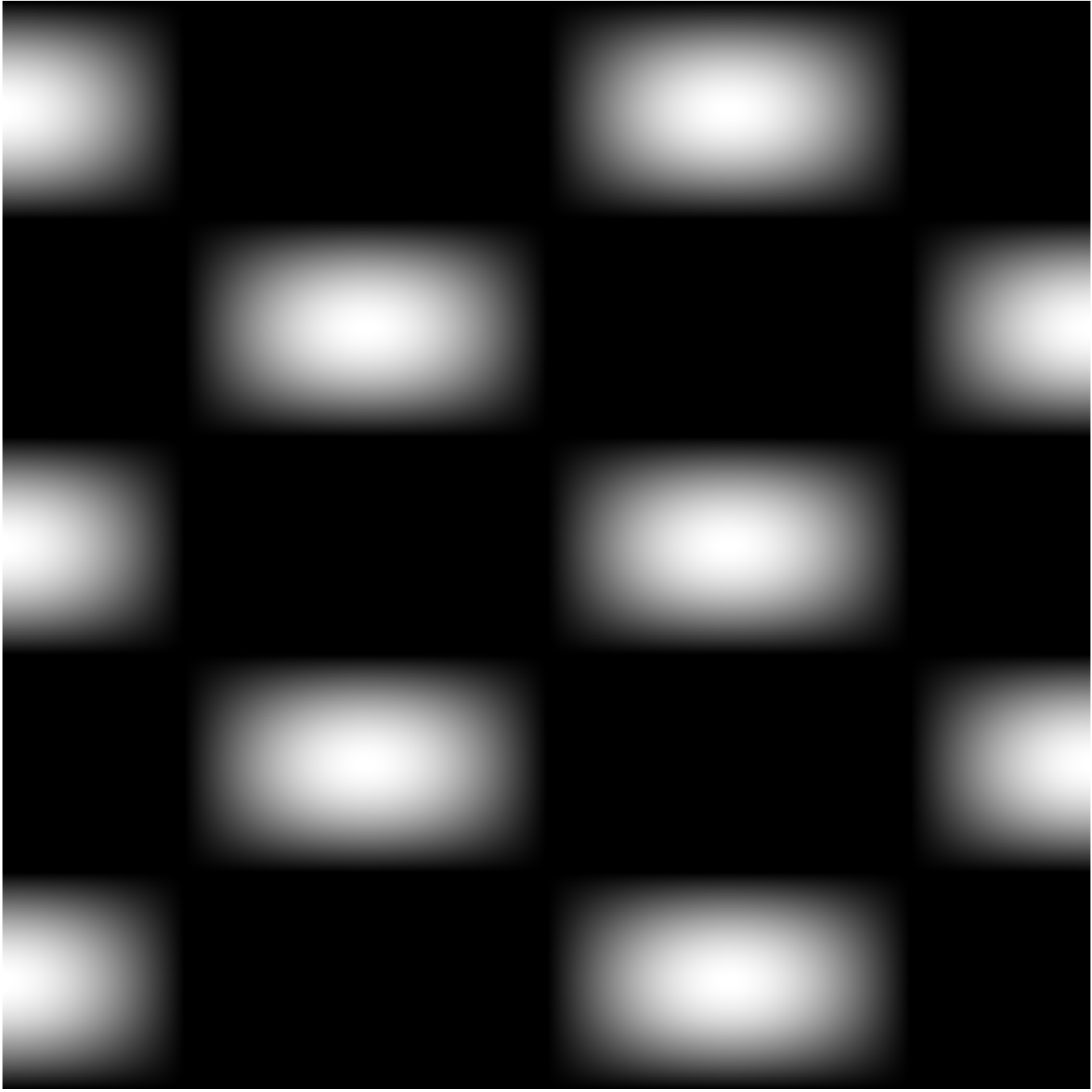} & \includegraphics[scale=0.238, trim=0cm 0cm 0cm 0cm]{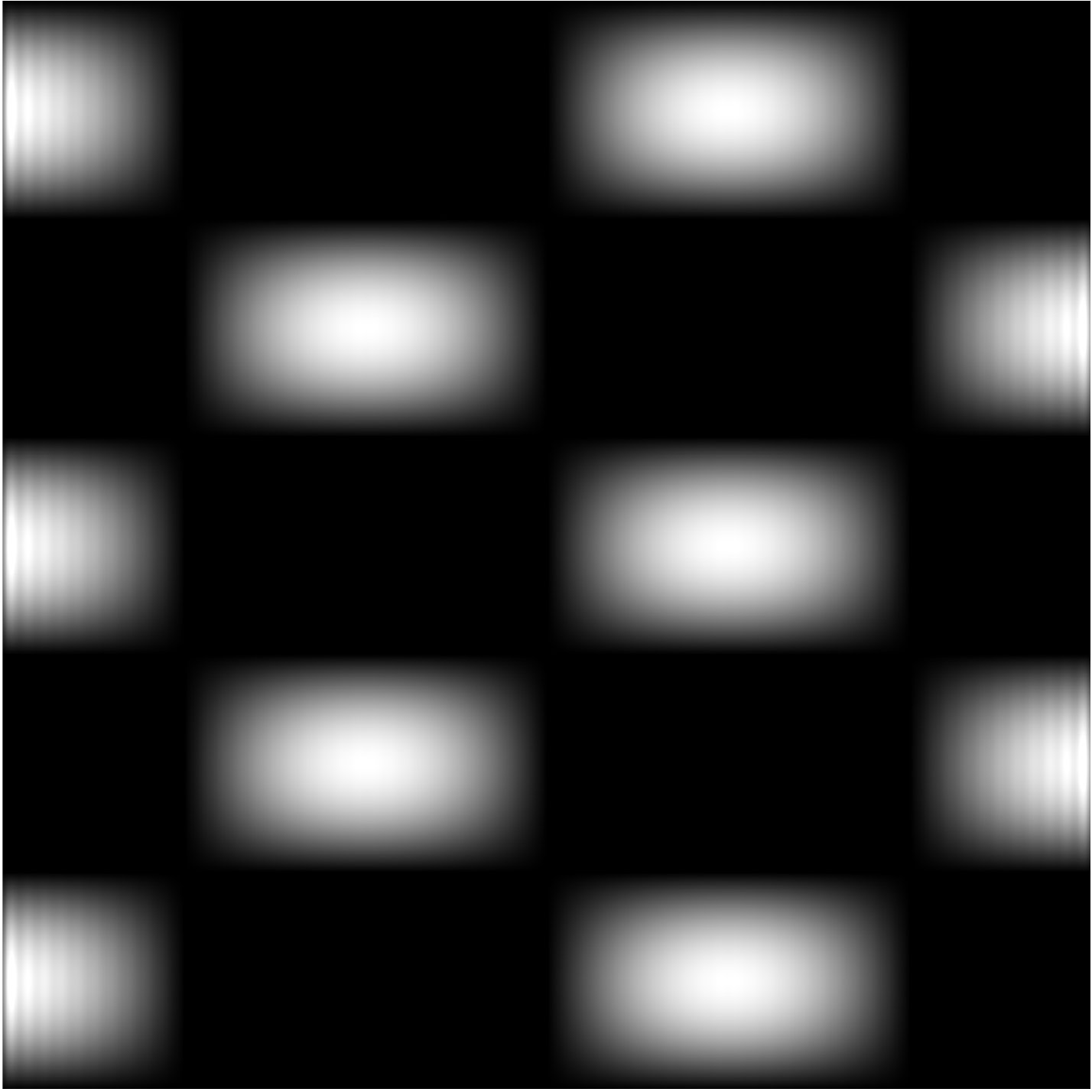}\\
 & \text{\small{$\|f-\tilde{f}\|\approx 4.13\times10^{-2}$}} & \text{\small{$\|f-\tilde{f}\|\approx 3.74\times10^{-3}$}} & \text{\small{$\|f-\tilde{f}\|\approx7.96\times10^{-4}$}} & \text{\small{$\|f-\tilde{f}\|\approx1.98\times10^{-2}$}}  \\
{\rotatebox{90}{\qquad\quad\text{\small{no weights}}}} & \includegraphics[scale=0.238, trim=0cm 0cm 0cm -0.25cm]{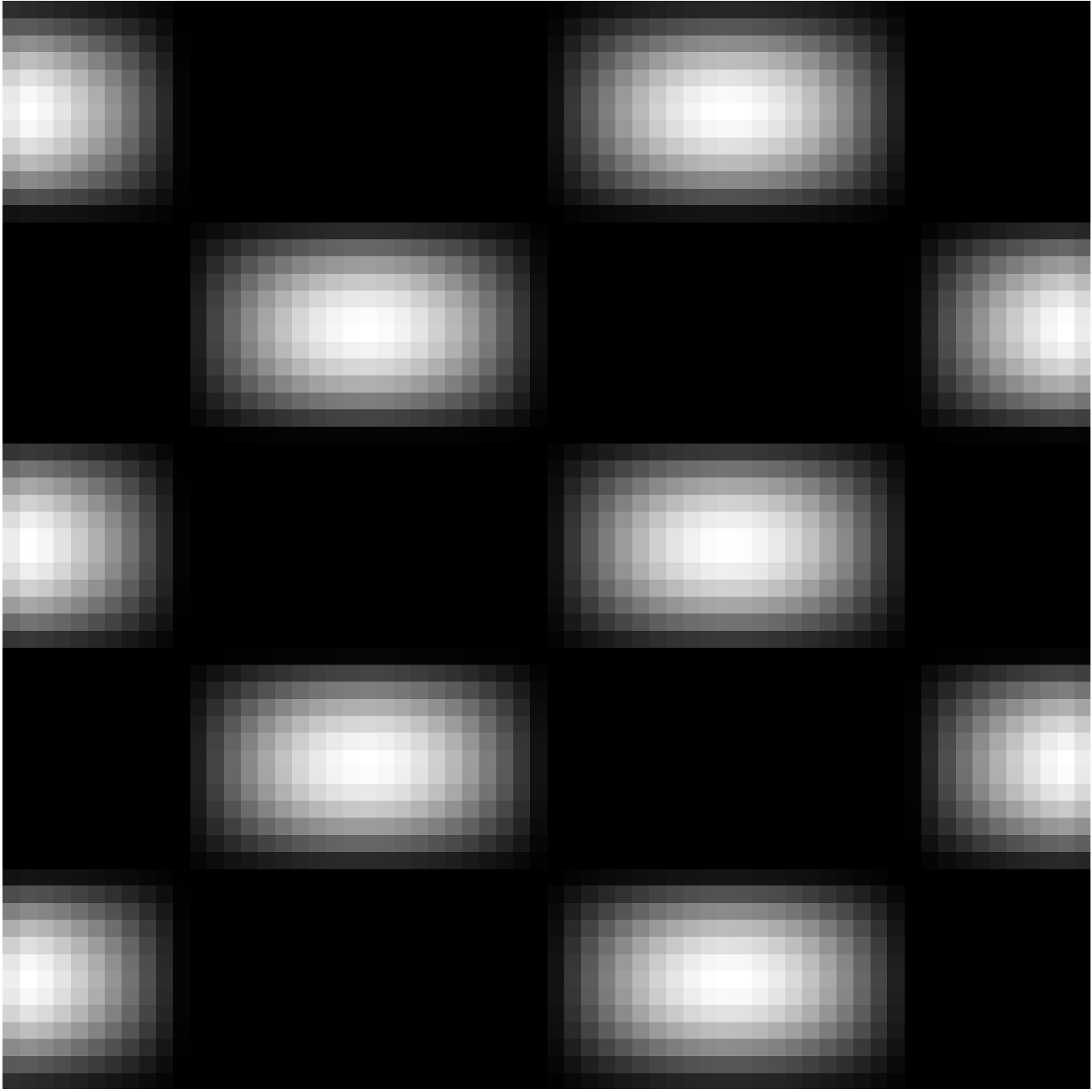} & \includegraphics[scale=0.238, trim=0cm 0cm 0cm -0.25cm]{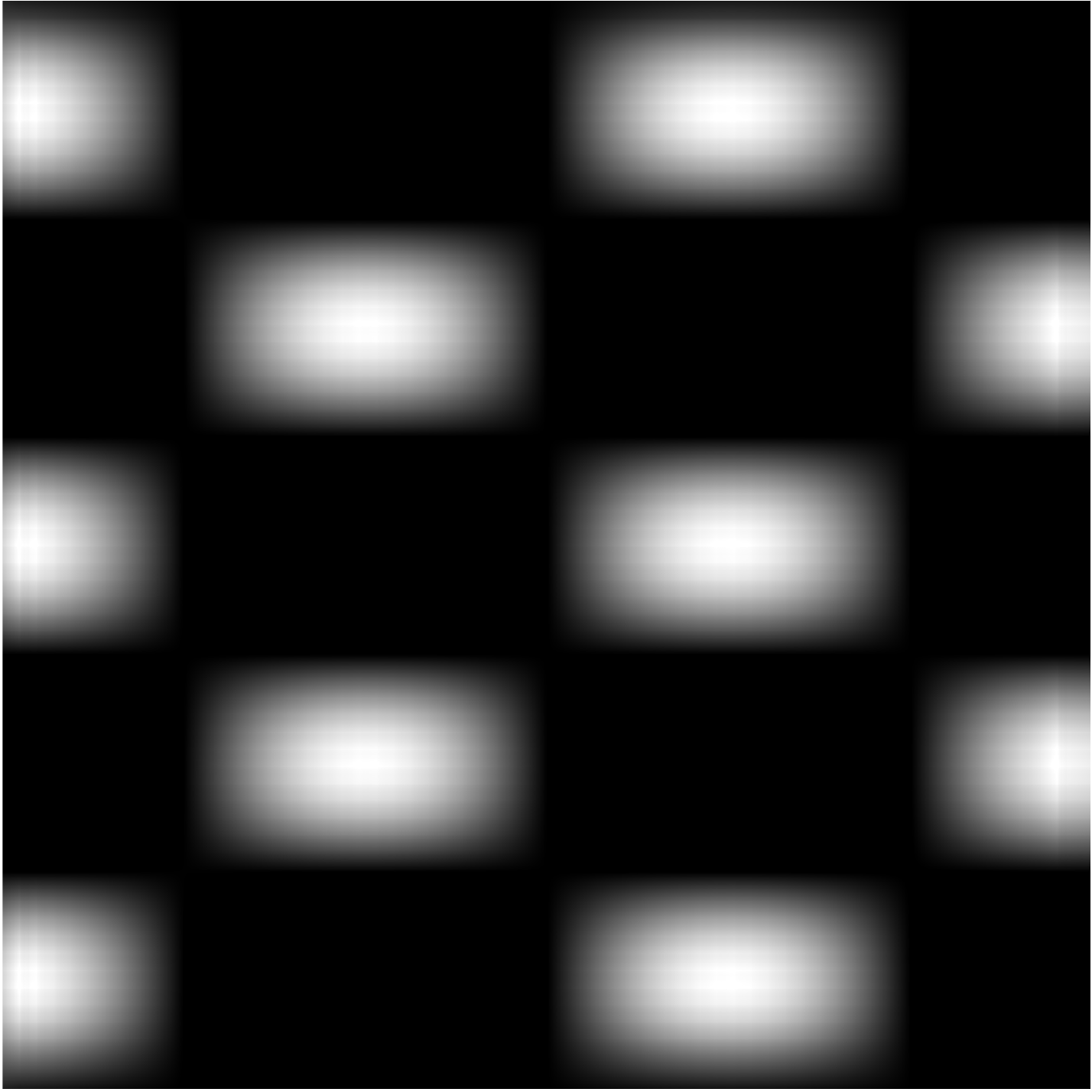} & \includegraphics[scale=0.238, trim=0cm 0cm 0cm -0.25cm]{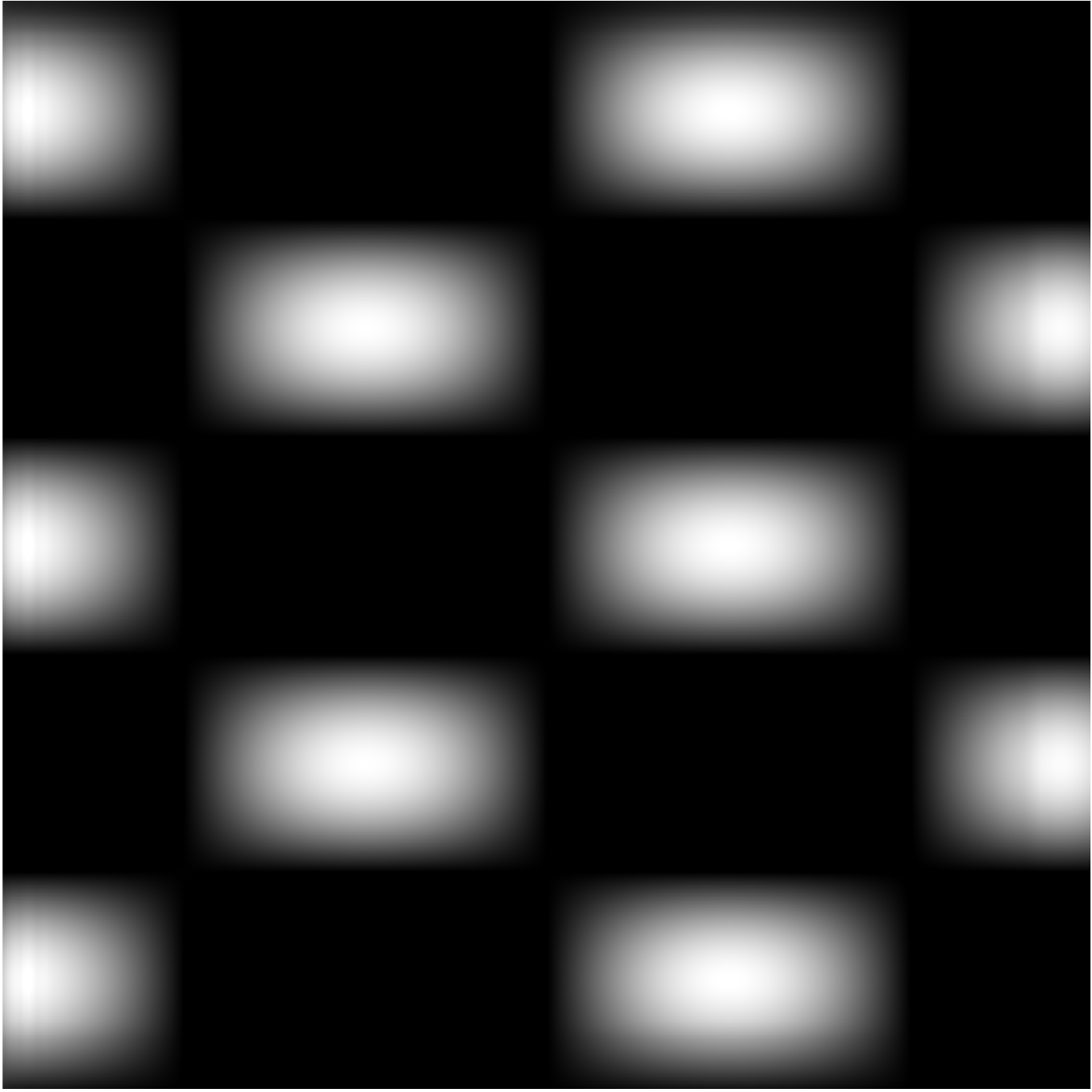} & \includegraphics[scale=0.238, trim=0cm 0cm 0cm -0.25cm]{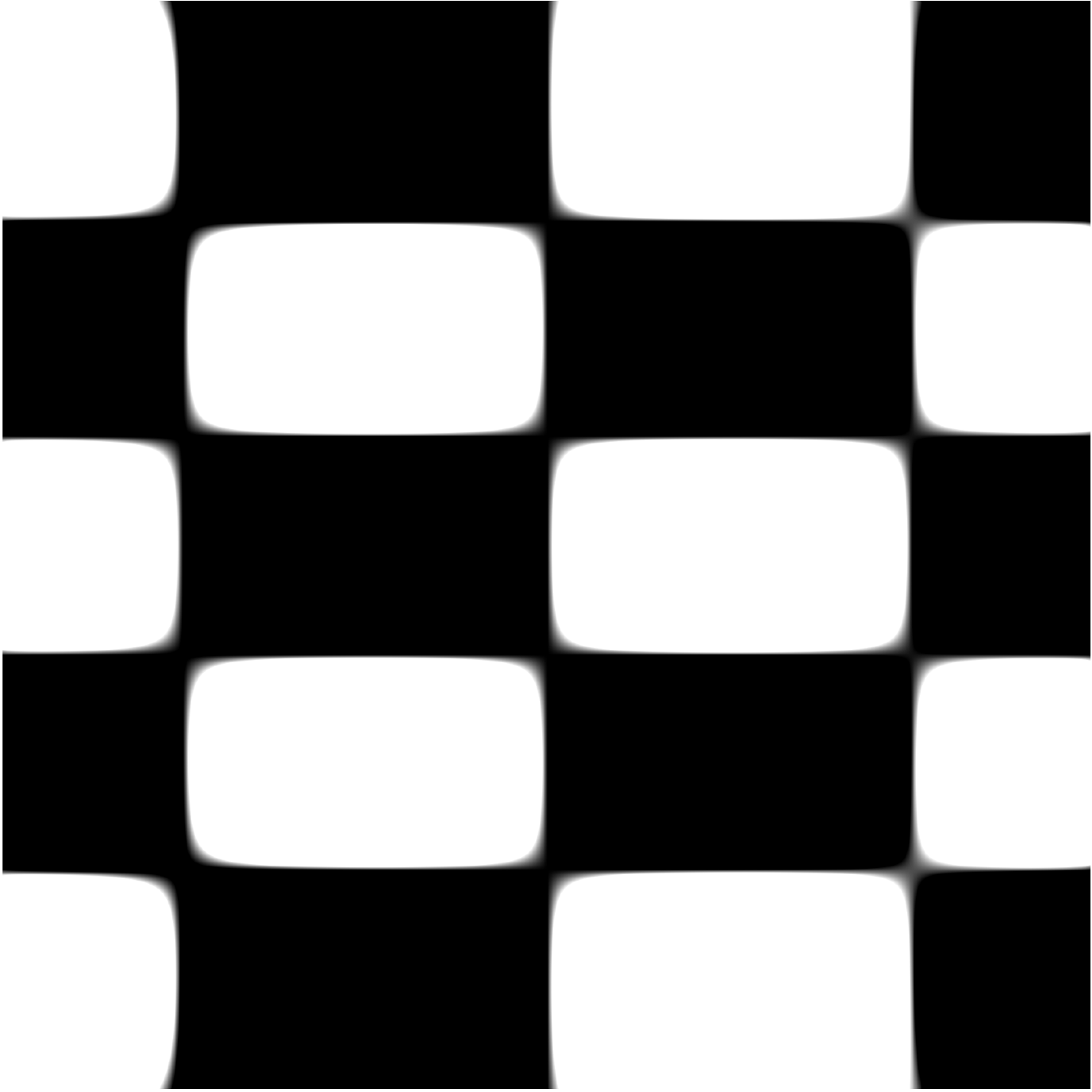} \\
& \text{\small{$\|f-\tilde{f}\|\approx4.32\times 10^{-2}$}} & \text{\small{$\|f-\tilde{f}\|\approx1.41\times10^{-2}$}} & \text{\small{$\|f-\tilde{f}\|\approx1.42\times10^{-2}$}} & \text{\small{$\|f-\tilde{f}\|\approx1.13\times10$}}  
\end{array}$
\caption{\small{Reconstructions of the function $f(x,y)=\sin(5/2\pi (x+1))\cos(3/2\pi (y+1)){\bf1}_{[-1,1]^2(x,y)}$ from Fourier samples taken on the radial sampling scheme in the Euclidean ball of radius $K=32$ with the density measured in $\ell^1$-norm strictly less than $1/4$. The lower pictures are reconstructed without using weights and, as demonstrated, the $\rL^2$-error does not exceed order $10^{-2}$. The NUGS reconstruction is computed with $64\times 64$ Haar, DB2 or DB3 wavelets.}}
\label{fig:weights_new}
\end{figure}

As noted earlier, NUGS with Haar wavelets is essentially equivalent to the iterative algorithms such as the one found in \cite{FesslerFastIterativeMRI} that use pixel basis.  As demonstrated in Figure \ref{fig:weights_new} for the two-dimensional setting (see \cite{1DNUGS} for univariate examples), the major advantage of NUGS is the possibility to change the approximation space $\rT$  and achieve better reconstructions.

Next, in Figure \ref{fig:density_det}, we examine how violation of the density condition given in Theorem \ref{t:weighted_frame} and part II of Theorem \ref{t:admissible} influences reconstruction of a high resolution test image. We use polar sampling schemes with different number of radial lines $n$ along which samples are acquired. Recall that the density condition from Theorem \ref{t:weighted_frame} is only sufficient, but not necessary to have a weighted Fourier frame, and that it is sharp in the sense that there exist a set of sampling points with $\delta_{E^{\circ}}=\delta_1=1/4$  and a function which violate the frame condition. Yet for a fixed function and set of sampling points, a slight violation of the density condition may not worsen the recovery guaranteed by the II part of Theorem \ref{t:admissible}.  As evident in the presented example from Figure \ref{fig:density_det}, a slight violation of $\delta_{E^{\circ}}<1/4$ does not impair the recovery noticeably therein.
However, it is evident that further decreasing of number of radial lines $n$, i.e.~decreasing of sampling density, worsens the quality of the reconstructed image. Also, as illustrated in Table \ref{tab:density_cond_num_t}, this decreasing of sampling density, i.e.~increasing of $\delta$, causes blowing up of the condition number associated to the least-squares system \R{GS_LS_data}.

\begin{table}[H]
\centering
{
\begin{tabular}{ |c||c|c|c|c|c|c| }
\hline
$n$ & 345 & 173 & 87 & 44 & 22 & 11 \\ \hline \hline
$\delta_2$ & 0.1763 & 0.3064  & 0.5847 & 1.1437 & 2.2843 & 4.5547 \\ \hline 
$\kappa$ & 1.6220 & 2.3821 & $1.4859\times10^{3}$ & $9.2459\times10^{14}$ & $5.3376\times10^{16}$ & $5.4891\times10^{18}$ \\ \hline
\end{tabular}
}
\caption{\small{The condition number $\kappa$ of a reconstruction matrix arising from \R{GS_LS_data} is calculated when $88\times88$ indicator functions are used and samples are acquired on a polar sampling scheme contained in $[-K,K]^2$, $K=32$, so that $\text{dim}(\rT)=(2.75K)^2$. The number of radial lines $n$ of the polar scheme is varying, as well as the corresponding sampling density $\delta_2$, which is measured with respect to the Euclidean norm.}  
}
\label{tab:density_cond_num_t}
\end{table}

\begin{figure}[H]
\begin{center}
  $\begin{array}{ccc}
  \multicolumn{3}{c}{\text{\small{Original image}} \hspace{2.2cm} \text{\small{Reconstruction}}}\\
  \multicolumn{3}{c}{\hspace{4.7cm} \text{\small{$n=1380$, $\delta_1<0.25$}}}\\
  \multicolumn{3}{c}{\includegraphics[scale=0.191,trim=0.cm 0cm 0cm 0cm]{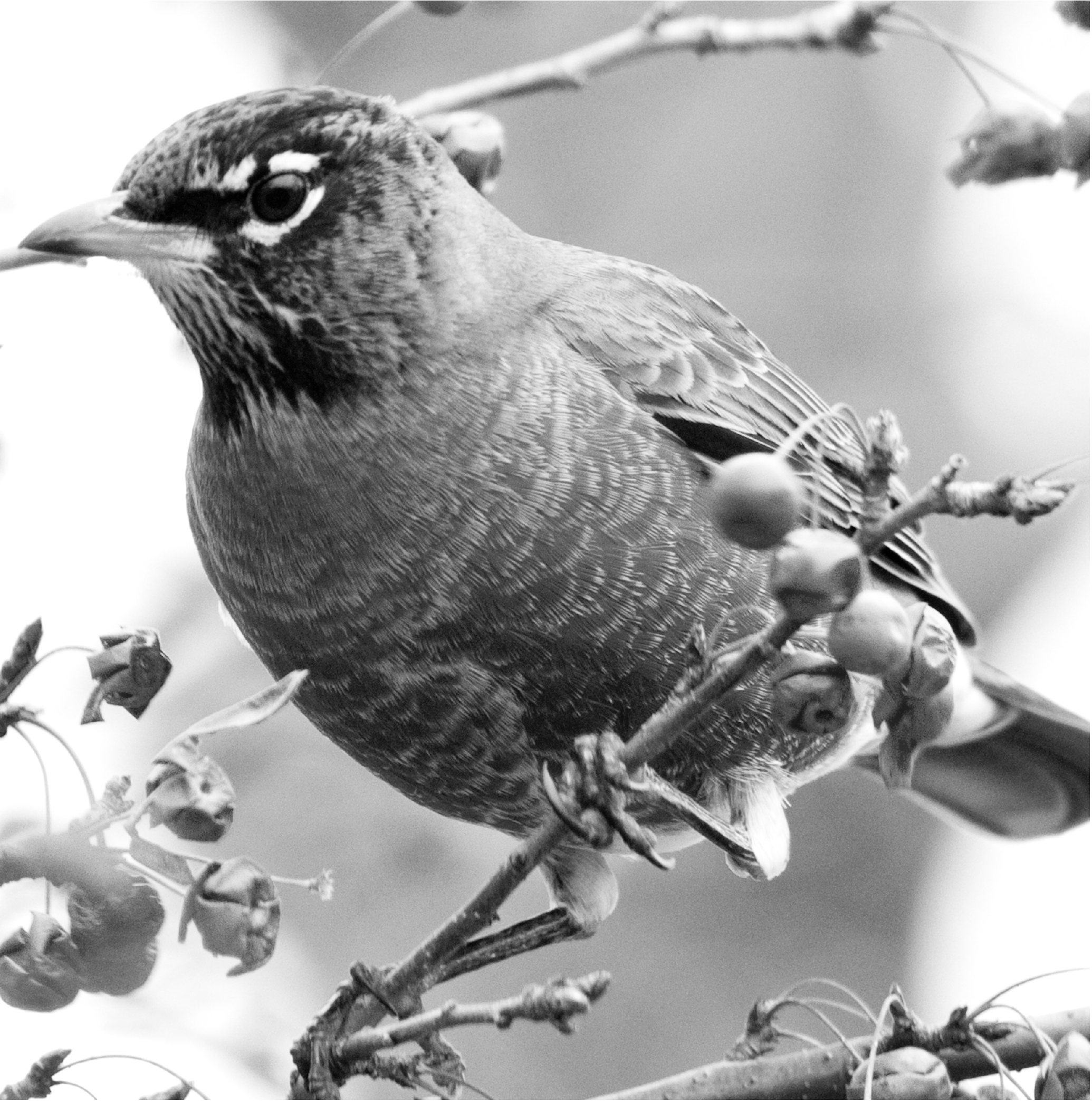} \quad \includegraphics[scale=0.44,trim=0.cm 0cm 0cm 0cm]{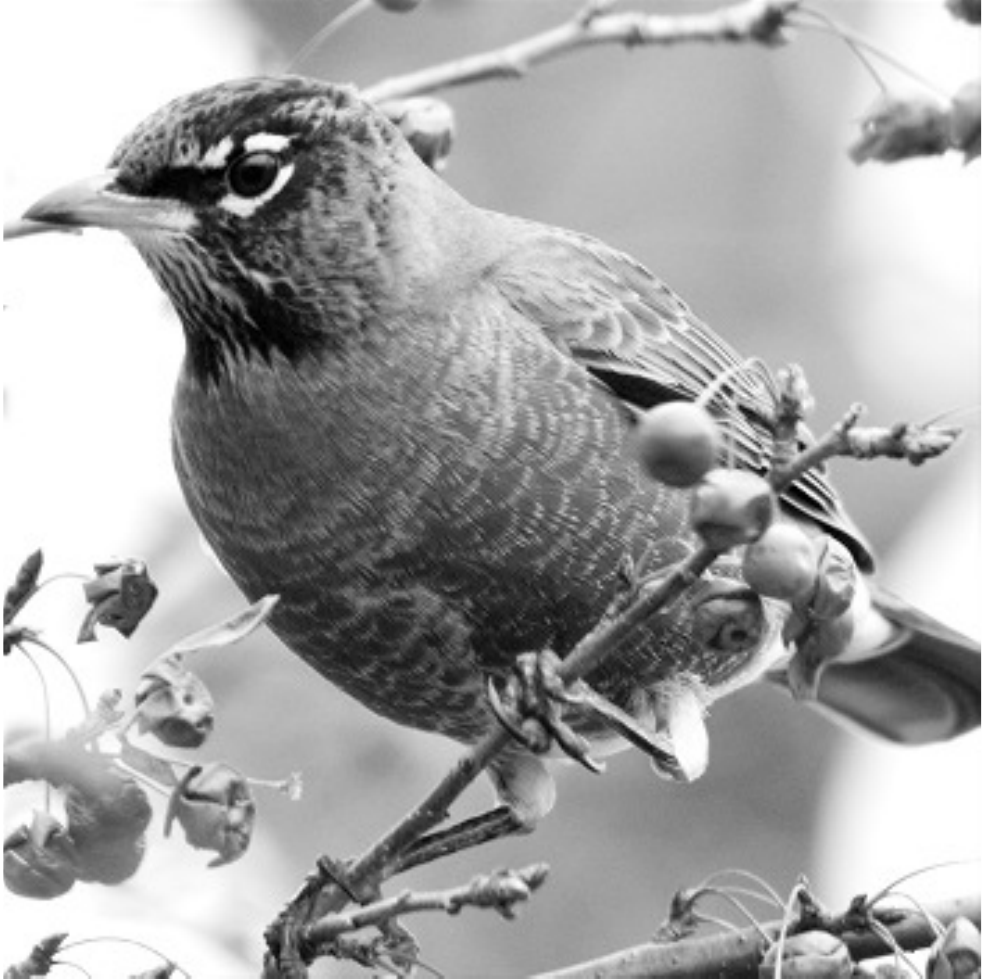}}\\
  \multicolumn{3}{c}{\text{\small{Reconstructions with insufficient densities}}}\\
  \text{\small{$n=690$, $\delta_2=0.31$}} &\text{\small{$n=345$, $\delta_2=0.59$}} & \text{\small{$n=173$, $\delta_2=1.17$}} \\  
  \includegraphics[scale=0.44,trim=0.cm 0cm 0cm 0cm]{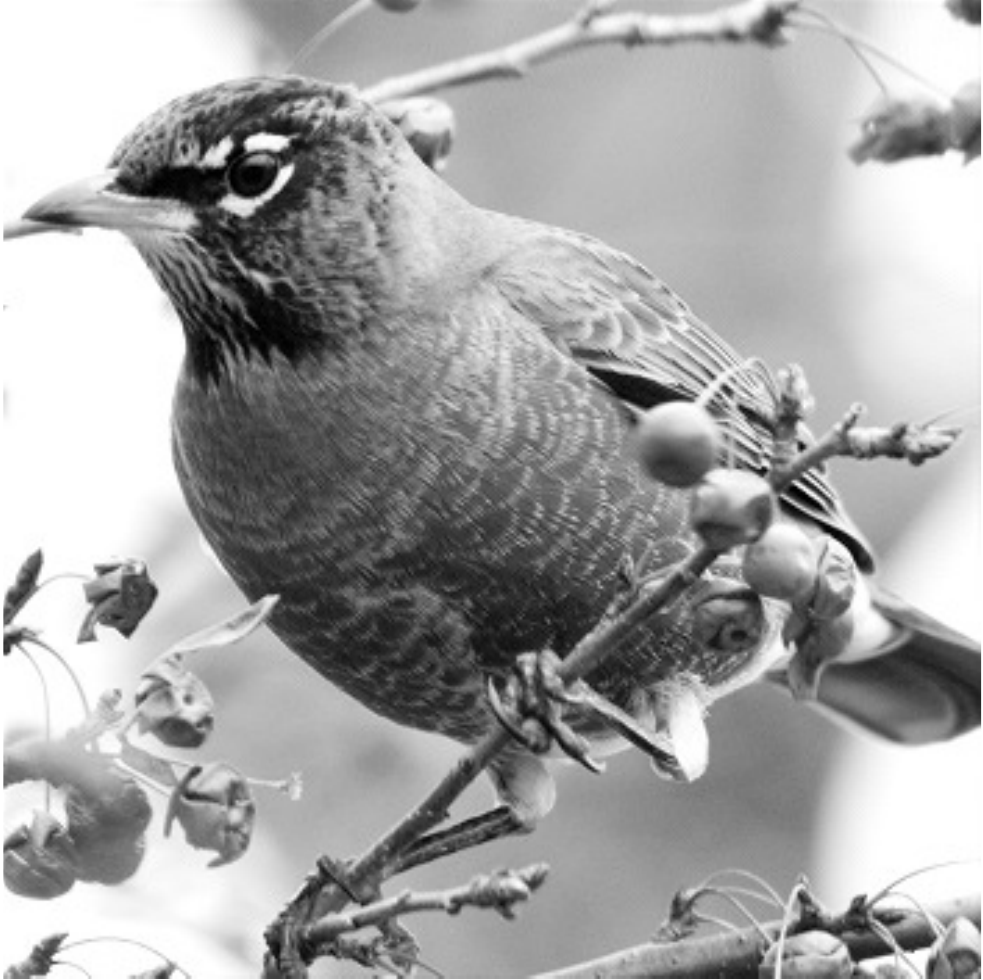} & \includegraphics[scale=0.44,trim=0.cm 0cm 0.cm 0cm]{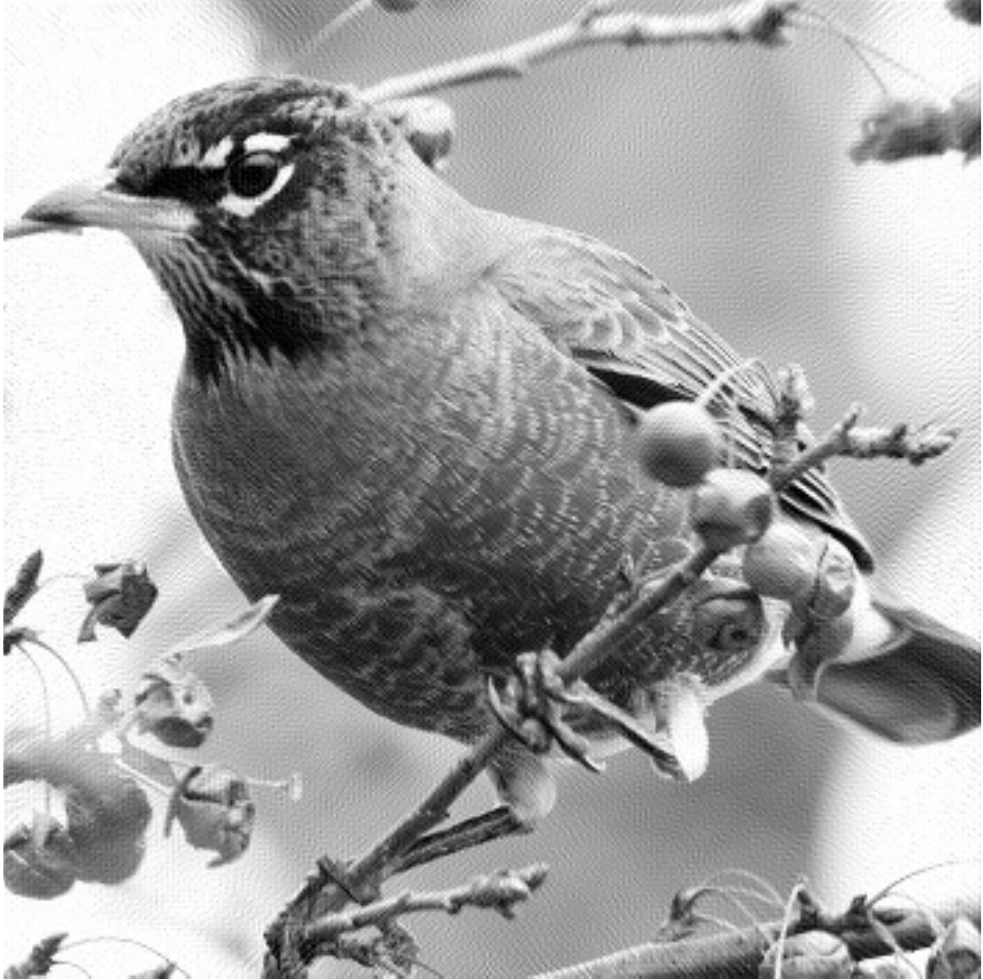} & \includegraphics[scale=0.44,trim=0.cm 0cm 0cm 0cm]{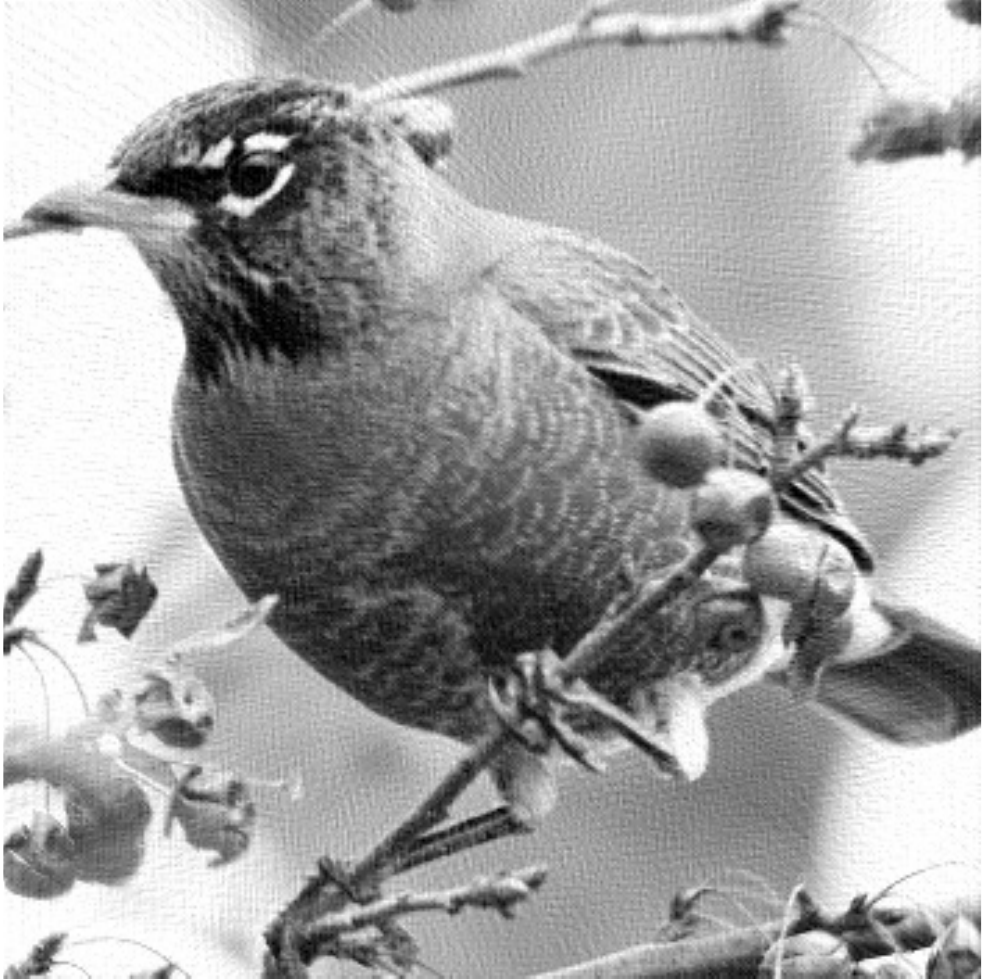} \\  
  \text{\small{$n=87$, $\delta_2=2.31$}} & \text{\small{$n=44$, $\delta_2=4.57$}} & \text{\small{$n=22$, $\delta_2=9.13$}} \\
  \includegraphics[scale=0.44,trim=0.cm 0cm 0cm 0cm]{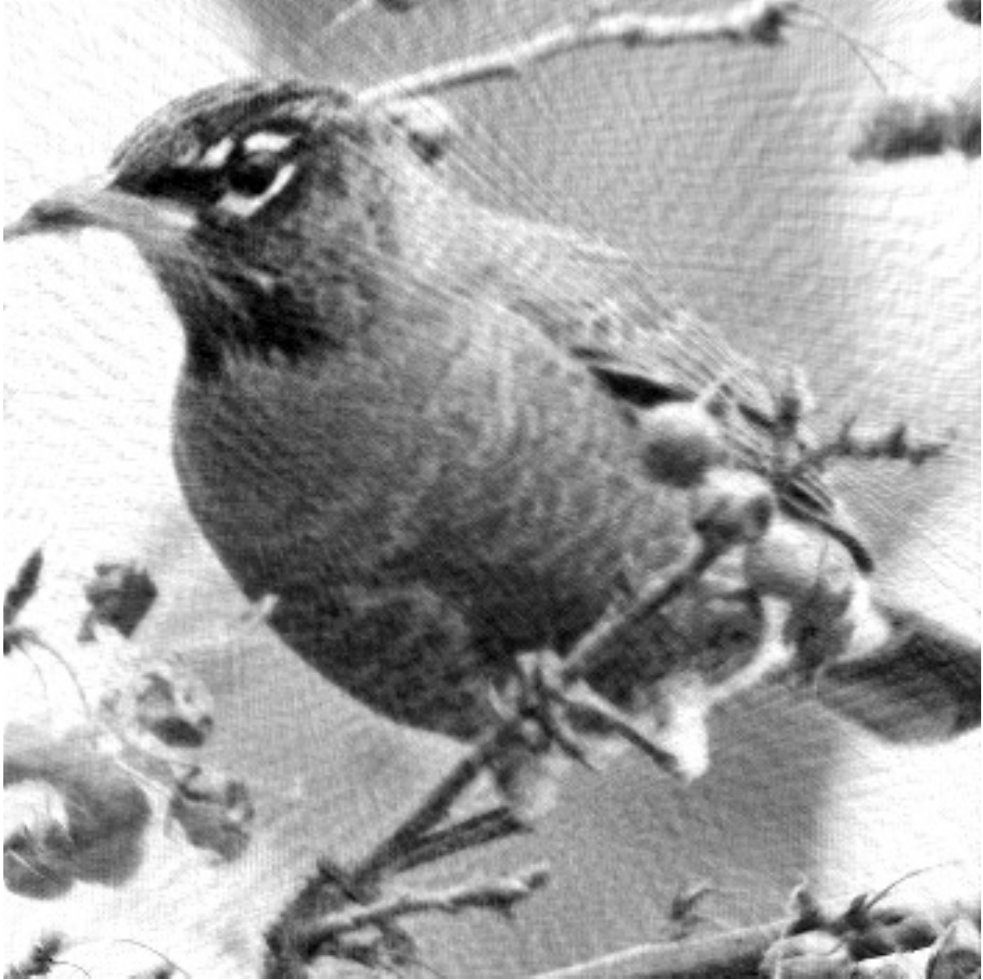} & \includegraphics[scale=0.44,trim=0.cm 0cm 0.cm 0cm]{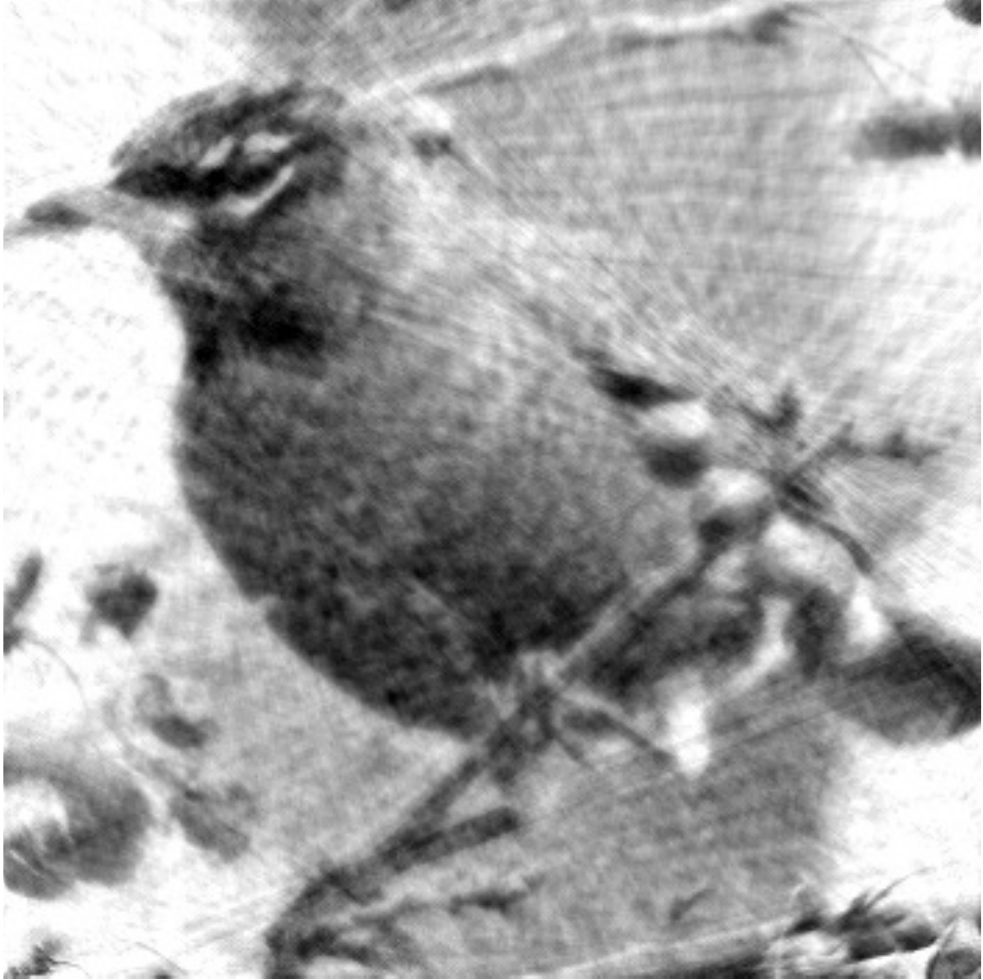} &  \includegraphics[scale=0.44,trim=0.cm 0cm 0.cm 0cm]{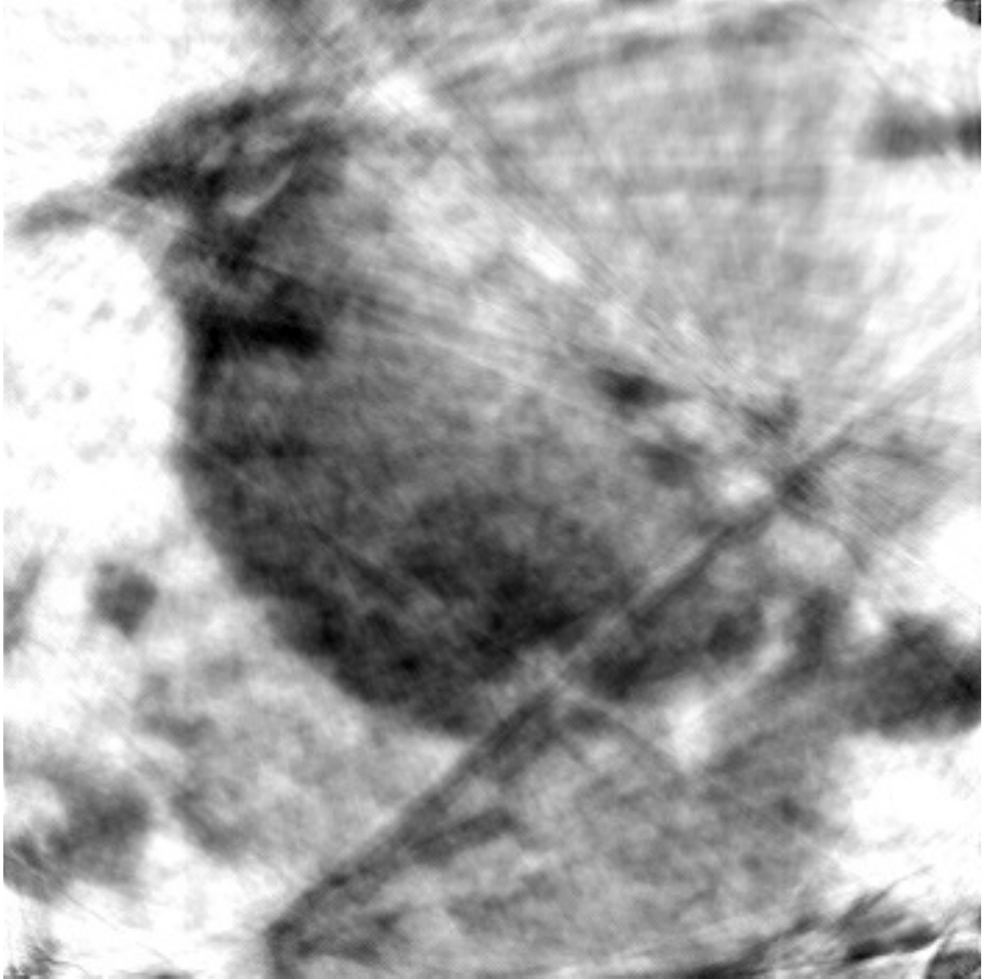}
  \end{array}$ 
    \vspace{-0.5cm}
\end{center}
\caption{\small{A high resolution image of $4500\times4500$ pixels is reconstructed by NUGS in the space $\rT$ consisting of $352\times352$ indicator functions when samples are taken on a polar sampling scheme contained in  $[-K,K]^2$, $K=128$. The relation $\text{dim}(\rT)=(2.75K)^2$ is used. The reconstructions are shown for sampling schemes with different densities, i.e.~different number of radial lines $n$. Here, the density in the Euclidean norm $\delta_2$ was directly computed on a computer. Since $\delta_1\geq\delta_2$, note that  $\delta_2\geq0.25$  ensures that the density condition $\delta_1<0.25$ is violated. }
    } 
    \label{fig:density_det}
\end{figure}

\section{Conclusions}\label{s:conclusions}

In the paper, we provide new theoretical insight of when a given countable set of sampling points yields a weighted Fourier frame, and therefore permits a multidimensional function recovery. To have a weighted Fourier frame for the space of $\rL^2$ functions supported on a compact convex and symmetric set $E$, it is enough to take pointwise measurements of its Fourier transform at points with density  $\delta_{E^\circ}<1/4$. Separation of sampling points is not required. Moreover, the weighted Fourier frame bounds are explicitly estimated in the case of smaller densities  than previously known, and in particular, their dimension dependence is removed for the space of functions supported on spheres. However, it remains an open problem to explicitly estimate frame bounds for even smaller densities (larger $\delta$), closer to the dimensionless condition $\delta_{E^\circ}<1/4$.

By exploiting these novel results on weighted Fourier frames,  the method for recovering a function in any given finite-dimensional space, known as NUGS, is analysed in multivariate setting. Its  stability and accuracy are guaranteed provided that finitely many samples are taken with both density and  bandwidth large enough. The density required is the same as the
one that guarantees weighted Fourier frames.

It remains an open question how to choose the sampling bandwidth $K$ depending on the specific reconstruction space. In \cite{1DNUGS}, the authors considered important case of reconstruction spaces $\rT$ consisting of compactly supported wavelets in the one-dimensional setting.  For any $\epsilon > 0$, it was shown that $R_K(\rT) < \epsilon$, provided $K \geq c(\epsilon) M$, where $M = \dim(\rT)$ and $c(\epsilon) >0$ is a constant depending on $\epsilon$ only (see \cite[Thm.\  5.3 and Thm.\  5.4]{1DNUGS}).  This means that a \textit{linear} scaling of the sampling bandwidth $K$ with the wavelet dimension $M$ is sufficient for stable recovery (necessity was also shown -- see \cite[Thm.\  6.1]{1DNUGS}). For this reason, wavelets subspaces are up to constant factors optimal spaces for reconstruction. These results from \cite{1DNUGS} present a generalization of the results proven in \cite{AHPWavelet} to the case of nonuniform Fourier samples. The case of wavelet recovery from uniform Fourier samples was extended to the multivariate setting in \cite{BAACHGKJM2Dwavelets}. We also expect these results to extend to the nonuniform multivariate case, but this is left for further investigations.

Having developed the NUGS framework in multivariate setting, it is possible to  consider  recoveries from nonuniform samples in any finite-dimensional space one desires. Besides wavelets, one can consider spaces consisting of  algebraic or trigonometric polynomials as they were considered in \cite{AdcockGataricHansenProceedings2015} in the one-dimensional case, as well as important generalizations of wavelets, such as curvelets and shearlets. This  is also left for future work.

\section*{Acknowledgements}
The authors would like to thank Karlheinz Gr\"ochenig and Gil Ramos for useful discussions and Clarice Poon for Matlab code used in an initial stage of our implementation.

\small
\bibliographystyle{plain}
\bibliography{references2D}

\begin{thebibliography}{10}

\bibitem{1DNUGS}
B.~Adcock, M.~Gataric, and A.~C. Hansen.
\newblock On stable reconstructions from nonuniform {F}ourier measurements.
\newblock {\em SIAM J. Imaging Sci.}, 7(3):1690--1723, 2014.

\bibitem{AdcockGataricHansenProceedings2015}
B.~Adcock, M.~Gataric, and A.~C. Hansen.
\newblock Recovering piecewise smooth functions from nonuniform {F}ourier
  measurements.
\newblock {\em \textnormal{Accepted in:} Proceedings of the 10th International
  Conference on Spectral and High Order Methods\textnormal{, and to be
  published in:} Springer Lecture Notes}, 2015.

\bibitem{BAACHShannon}
B.~Adcock and A.~C. Hansen.
\newblock A generalized sampling theorem for stable reconstructions in
  arbitrary bases.
\newblock {\em J. Fourier Anal. Appl.}, 18(4):685--716, 2012.

\bibitem{BAACHAccRecov}
B.~Adcock and A.~C. Hansen.
\newblock Stable reconstructions in {H}ilbert spaces and the resolution of the
  {G}ibbs phenomenon.
\newblock {\em Appl. Comput. Harmon. Anal.}, 32(3):357--388, 2012.

\bibitem{BAACHGKJM2Dwavelets}
B.~Adcock, A.~C. Hansen, G.~Kutyniok, and J.~Ma.
\newblock Linear stable sampling rate: Optimality of 2d wavelet reconstructions
  from fourier measurements.
\newblock {\em SIAM Journal on Mathematical Analysis}, 47(2):1196--1233, 2015.

\bibitem{BAACHOptimality}
B.~Adcock, A.~C. Hansen, and C.~Poon.
\newblock Beyond consistent reconstructions: optimality and sharp bounds for
  generalized sampling, and application to the uniform resampling problem.
\newblock {\em SIAM J. Math. Anal.}, 45(5):3114--3131, 2013.

\bibitem{AHPWavelet}
B.~Adcock, A.~C. Hansen, and C.~Poon.
\newblock On optimal wavelet reconstructions from {F}ourier samples: linearity
  and universality of the stable sampling rate.
\newblock {\em Appl. Comput. Harmon. Anal.}, 36(3):387--415, 2014.

\bibitem{AldroubiAverageSamp}
A.~Aldroubi.
\newblock Non-uniform weighted average sampling and reconstruction in
  shift-invariant and wavelet spaces.
\newblock {\em Appl. Comput. Harmon. Anal.}, 13:151--161, 2002.

\bibitem{AldroubiGrochenig_SIdensity}
A.~Aldroubi and K.~Gr{\"o}chenig.
\newblock Beurling-{L}andau-type theorems for non-uniform sampling in shift
  invariant spline spaces.
\newblock {\em J. Fourier Anal. Appl.}, 6(1):93--103, 2000.

\bibitem{AldroubiGrochenigSIREV}
A.~Aldroubi and K.~Gr{\"o}chenig.
\newblock {Nonuniform sampling and reconstruction in shift-invariant spaces}.
\newblock {\em SIAM Rev.}, 43:585--620, 2001.

\bibitem{BassGrochenigRandom}
R.~F. Bass and K.~Gr{\"o}chenig.
\newblock Random sampling of multivariate trigonometric polynomials.
\newblock {\em SIAM J. Math. Anal.}, 36(3):773--795 (electronic), 2004/05.

\bibitem{BenedettoFrames}
J.~J. Benedetto.
\newblock Irregular sampling and frames.
\newblock In {\em Wavelets}, volume~2 of {\em Wavelet Anal. Appl.}, pages
  445--507. Academic Press, Boston, MA, 1992.

\bibitem{BenedettoSpiral}
J.~J. Benedetto and H.~C. Wu.
\newblock {Non-uniform sampling and spiral {MRI} reconstruction}.
\newblock {\em Proc. SPIE}, 4119:130--141, 2000.

\bibitem{BenedettoBook}
J.J. Benedetto and P.J.S.G. Ferreira.
\newblock {\em Modern Sampling Theory: Mathematics and Applications}.
\newblock Applied and Numerical Harmonic Analysis. Birkh{\"a}user Boston, 2001.

\bibitem{BeurlingDiffOp}
A.~Beurling.
\newblock Local harmonic analysis with some applications to differential
  operators.
\newblock In {\em Some {R}ecent {A}dvances in the {B}asic {S}ciences, {V}ol. 1
  ({P}roc. {A}nnual {S}ci. {C}onf., {B}elfer {G}rad. {S}chool {S}ci., {Y}eshiva
  {U}niv., {N}ew {Y}ork, 1962--1964)}, pages 109--125. Belfer Graduate School
  of Science, Yeshiva Univ., New York, 1966.

\bibitem{BeurlingVol2}
A.~Beurling.
\newblock {\em The collected works of {A}rne {B}eurling. {V}ol. 2}.
\newblock Contemporary Mathematicians. Birkh\"auser Boston Inc., Boston, MA,
  1989.
\newblock Harmonic analysis, Edited by L. Carleson, P. Malliavin, J. Neuberger
  and J. Wermer.

\bibitem{ChristensenFramesAMS}
O.~Christensen.
\newblock Frames, {R}iesz bases, and discrete {G}abor/wavelet expansions.
\newblock {\em Bull. Amer. Math. Soc}, 38(3):273--291, 2001.

\bibitem{spyralScienceDirect}
B.~M.~A. Delattre, R.~M. Heidemann, L.~A. Crowe, J.-P. Vall\'{e}e, and J.-N.
  Hyacinthe.
\newblock {Spiral demystified}.
\newblock {\em Magn. Reson. Imaging}, 28(6):862--881, 2010.

\bibitem{DuffSch}
R.~J. Duffin and A.~C. Schaeffer.
\newblock A class of nonharmonic {F}ourier series.
\newblock {\em Trans. Amer. Math. Soc.}, 72:341--366, 1952.

\bibitem{eldar2003sampling}
Y.~C. Eldar.
\newblock Sampling without input constraints: Consistent reconstruction in
  arbitrary spaces.
\newblock In A.~I. Zayed and J.~J. Benedetto, editors, {\em Sampling, Wavelets
  and Tomography}, pages 33--60. Boston, MA: Birkh{\"a}user, 2004.

\bibitem{eldar2005general}
Y.~C. Eldar and T.~Werther.
\newblock General framework for consistent sampling in {H}ilbert spaces.
\newblock {\em Int. J. Wavelets Multiresolut. Inf. Process.}, 3(3):347, 2005.

\bibitem{Epstein}
C.~L. Epstein.
\newblock {\em Introduction to the mathematics of medical imaging}.
\newblock Society for Industrial and Applied Mathematics (SIAM), Philadelphia,
  PA, second edition, 2008.

\bibitem{FeichtingerGrochenigIrregular}
H.~G. Feichtinger and K.~Gr{\"o}chenig.
\newblock {Theory and practice of irregular sampling}.
\newblock In J.~J. Benedetto and M.~Frazier, editors, {\em {Wavelets:
  Mathematics and Applications}}, pages 305--363. Boca Raton, FL: CRC, 1994.

\bibitem{FeichtingerEtAlEfficientNonuniform}
H.~G. Feichtinger, K.~Gr{\"o}chenig, and T.~Strohmer.
\newblock {Efficient numerical methods in nonuniform sampling theory}.
\newblock {\em Numer. Math.}, 69:423--440, 1995.

\bibitem{FesslerNUFFT}
J.~A. Fessler and B.~P. Sutton.
\newblock Nonuniform fast {F}ourier transforms using min-max interpolation.
\newblock {\em IEEE Trans. Signal Process.}, 51(2):560--574, 2003.

\bibitem{Gabardo}
J.-P. Gabardo.
\newblock Weighted tight frames of exponentials on a finite interval.
\newblock {\em Monatsh. Math.}, 116(3-4):197--229, 1993.

\bibitem{GataricPoon2015}
M.~Gataric and C.~Poon.
\newblock A practical guide to the recovery of wavelet coefficients from
  {F}ourier measurements.
\newblock {\em Preprint}, 2015.

\bibitem{GrochenigIrregular}
K.~Gr{\"o}chenig.
\newblock {Reconstruction algorithms in irregular sampling}.
\newblock {\em Math. Comp.}, 59:181--194, 1992.

\bibitem{GrochenigTrigonometric}
K.~Gr{\"o}chenig.
\newblock Irregular sampling, {T}oeplitz matrices, and the approximation of
  entire functions of exponential type.
\newblock {\em Math. Comp.}, 68(226):749--765, 1999.

\bibitem{GrochenigModernSamplingBook}
K.~Gr\"{o}chenig.
\newblock Non-uniform sampling in higher dimensions: from trigonometric
  polynomials to bandlimited functions.
\newblock In J.~J. Benedetto, editor, {\em Modern Sampling Theory}, chapter~7,
  pages 155--171. Birkh\"{o}user Boston, 2001.

\bibitem{GrochenigStrohmerMarvasti}
K.~Gr\"{o}chenig and T.~Strohmer.
\newblock Numerical and theoretical aspects of non-uniform sampling of
  band-limited images.
\newblock In F.~Marvasti, editor, {\em Nonuniform Sampling: Theory and
  Applications}, chapter~6, pages 283--324. Kluwer {A}cademic, {D}ordrecht,
  {T}he {N}etherlands, 2001.

\bibitem{PruessmannUnserMRIFast}
M.~Guerquin-Kern, M.~H{\"a}berlin, K.~P. Pruessmann, and M.~Unser.
\newblock {A fast wavelet-based reconstruction method for {M}agnetic
  {R}esonance {I}maging}.
\newblock {\em IEEE Trans. Med. Imaging}, 30(9):1649--1660, 2011.

\bibitem{hrycakIPRM}
T.~Hrycak and K.~Gr\"{o}chenig.
\newblock Pseudospectral {F}ourier reconstruction with the modified inverse
  polynomial reconstruction method.
\newblock {\em J. Comput. Phys.}, 229(3):933--946, 2010.

\bibitem{JacksonEtAlGridding}
J.~I. Jackson, C.~H. Meyer, D.~G. Nishimura, and A.~Macovski.
\newblock {Selection of a convolution function for {F}ourier inversion using
  gridding}.
\newblock {\em IEEE Trans. Med. Imaging}, 10:473--478, 1991.

\bibitem{Jaffard}
S.~Jaffard.
\newblock {A density criterion for frames of complex exponentials}.
\newblock {\em Mich. Math. J.}, 38(3):339--348, 1991.

\bibitem{PolynomialJung}
J.-H. Jung and B.~D. Shizgal.
\newblock Generalization of the inverse polynomial reconstruction method in the
  resolution of the {G}ibbs phenomenon.
\newblock {\em J. Comput. Appl. Math.}, 172(1):131--151, 2004.

\bibitem{PottsNUFFT}
J.~Keiner, S.~Kunis, and D.~Potts.
\newblock {Using NFFT 3---A Software Library for Various Nonequispaced Fast
  Fourier Transforms}.
\newblock {\em ACM Trans. Math. Softw.}, 36(4):19:1--19:30, 2009.

\bibitem{Klein}
R.~Klein.
\newblock {\em Concrete and abstract {V}orono\u\i\ diagrams}, volume 400 of
  {\em Lecture Notes in Computer Science}.
\newblock Springer-Verlag, Berlin, 1989.

\bibitem{KnoppKunisPotts}
T.~Knopp, S.~Kunis, and D.~Potts.
\newblock A note on the iterative {MRI} reconstruction from nonuniform k-space
  data.
\newblock {\em Int. J. Bio. Imag.}, page 24727, 2007.

\bibitem{KunisPottsSIAM}
S.~Kunis and D.~Potts.
\newblock Stability results for scattered data interpolation by trigonometric
  polynomials.
\newblock {\em SIAM J. Sci. Comput.}, 29(4):1403--1419 (electronic), 2007.

\bibitem{Landau}
H.~J. Landau.
\newblock Necessary density conditions for sampling and interpolation of
  certain entire functions.
\newblock {\em Acta Math.}, 117:37--52, 1967.

\bibitem{Levinson}
N.~Levinson.
\newblock {\em Gap and {D}ensity {T}heorems}.
\newblock American Mathematical Society Colloquium Publications, v. 26.
  American Mathematical Society, New York, 1940.

\bibitem{Marvasti}
F.~Marvasti.
\newblock {\em Nonuniform Sampling: Theory and Practice}.
\newblock Number v. 1 in Information Technology Series. Springer US, 2001.

\bibitem{quasi}
B.~Matei and Y.~Meyer.
\newblock Simple quasicrystals are sets of stable sampling.
\newblock {\em Complex Var. Elliptic Equ.}, 55(8-10):947--964, 2010.

\bibitem{OlevskiiUlanovskii}
A.~Olevskii and A.~Ulanovskii.
\newblock On multi-dimensional sampling and interpolation.
\newblock {\em Anal. Math. Phys.}, 2(2):149--170, 2012.

\bibitem{PaleyWiener}
R.~E. A.~C. Paley and N.~Wiener.
\newblock {\em Fourier transforms in the complex domain}, volume~19 of {\em
  American Mathematical Society Colloquium Publications}.
\newblock American Mathematical Society, Providence, RI, 1987.
\newblock Reprint of the 1934 original.

\bibitem{PottsTasche}
D.~Potts and M.~Tasche.
\newblock {Numerical stability of nonequispaced fast Fourier transforms}.
\newblock {\em J. Comput. Appl. Math.}, 222(2):655--674, 2008.

\bibitem{pruessmann2001}
K.~P. Pruessmann, M.~Weiger, P.~B\"{o}rnert, and P.~Boesiger.
\newblock Advances in sensitivity encoding with arbitrary k-space trajectories.
\newblock {\em Magnetic Resonance in Medicine}, 46(4):638--651, 2001.

\bibitem{Rasche99}
V.~Rasche, R.~Proksa, R.~Sinkus, P.~Bornert, and H.~Eggers.
\newblock Resampling of data between arbitrary grids using convolution
  interpolation.
\newblock {\em IEEE Trans. Med. Imaging}, 18(5):385--392, 1999.

\bibitem{RosenfeldURS2}
D.~Rosenfeld.
\newblock {New approach to gridding using regularlization and estimation
  theory}.
\newblock {\em Magn. Reson. Med.}, 48(1):193--202, 2002.

\bibitem{Seip1995}
K.~Seip.
\newblock On the connection between exponential bases and certain related
  sequences in {$L^2(-\pi,\pi)$}.
\newblock {\em J. Funct. Anal.}, 130(1):131--160, 1995.

\bibitem{SeipBook}
K.~Seip.
\newblock {\em Interpolation and sampling in spaces of analytic functions},
  volume~33 of {\em University Lecture Series}.
\newblock American Mathematical Society, Providence, RI, 2004.

\bibitem{StrohmerNonuniformNA}
T.~Strohmer.
\newblock Numerical analysis of the non-uniform sampling problem.
\newblock {\em J. Comput. Appl. Math.}, 122(1--2):297--316, 2000.

\bibitem{FesslerFastIterativeMRI}
B.~P. Sutton, D.~C. Noll, and J.~A. Fessler.
\newblock {Fast, iterative image reconstruction for MRI in the presence of
  field inhomogeneities}.
\newblock {\em IEEE Trans. Med. Imaging}, 22(2):178--188, 2003.

\bibitem{Unser50years}
M.~Unser.
\newblock Sampling---50 {Y}ears after {S}hannon.
\newblock {\em Proceedings of the {IEEE}}, 88(4):569--587, 2000.

\bibitem{unser1994general}
M.~Unser and A.~Aldroubi.
\newblock A general sampling theory for nonideal acquisition devices.
\newblock {\em IEEE Trans. Signal Process.}, 42(11):2915--2925, 1994.

\bibitem{UnserAldroubiWaveletReview}
M.~Unser and A.~Aldroubi.
\newblock A review of wavelets in biomedical applications.
\newblock {\em Proc. IEEE}, 84(4):626--638, 1996.

\bibitem{Young}
R.~M. Young.
\newblock {\em {An Introduction to Nonharmonic Fourier Series}}.
\newblock Academic Press Inc., first edition, 2001.

\end{thebibliography}
 
\end{document}